\documentclass[a4, 10pt]{article}

\usepackage[a4paper,left=25mm,right=25mm,top=25mm,bottom=30mm]{geometry}
\usepackage{setspace}
\usepackage{amsmath,amssymb,amsthm}
\usepackage{mathtools}
\usepackage{ascmac}
\usepackage{bm}
\usepackage{graphicx}
\usepackage{hyperref}
\usepackage{color}
\usepackage{algpseudocode}
\usepackage{algorithm}
\usepackage{multicol,multirow}
\usepackage{boites}
\usepackage{framed}
\usepackage{setspace}
\usepackage{float}
\usepackage{xcolor}
\hypersetup{
	colorlinks=false,
	citebordercolor=green,
	linkbordercolor=red,
	urlbordercolor=cyan,
}

\usepackage[style=numeric,doi=false,eprint=true,url=false,isbn=false,maxbibnames=10]{biblatex}
\AtEveryBibitem{\ifentrytype{book}{\clearfield{pages}}{}}
\DeclareFieldFormat{pages}{#1}
\renewbibmacro{in:}{}
\theoremstyle{plain}
\newtheorem{theorem}{Theorem}
\newtheorem{prop}[theorem]{Proposition}
\newtheorem{lemma}[theorem]{Lemma}

\theoremstyle{definition}

\theoremstyle{remark}
\newtheorem{rem}[theorem]{Remark}
\newtheorem{example}[theorem]{Example}

\numberwithin{theorem}{section}
\numberwithin{equation}{section}

\newcommand\Z{\mathbb{Z}}

\newcommand\R{\mathbb{R}}

\newcommand\T{\mathbb{T}}

\newcommand\eps{\varepsilon}

\newcommand\qqquad{\qquad \qquad}
\newcommand\del{\partial}
\renewcommand\bar[1]{\overline{#1}}
\newcommand\sgn[1]{\mathrm{sgn}\,({#1})}

\definecolor{deeppink}{HTML}{ff1493}
\definecolor{royalblue}{HTML}{4169e1}
\definecolor{gainsboro}{HTML}{dcdcdc}

\newcommand{\new}[1]{#1} 

\title{\emph{On the convergence rates of generalized conditional gradient method for fully discretized Mean Field Games}}
\author{Haruka Nakamura\thanks{Graduate School of Mathematical Sciences, The University of Tokyo, Komaba 3-8-1, Meguro-ku, Tokyo 153-8914, Japan. \texttt{nakamura-haruka@g.ecc.u-tokyo.ac.jp}} \quad and \quad 
Norikazu Saito\thanks{Graduate School of Mathematical Sciences, The University of Tokyo, Komaba 3-8-1, Meguro-ku, Tokyo 153-8914, Japan. \texttt{norikazu@g.ecc.u-tokyo.ac.jp}}
}

\date{\today}

\begin{document}

\maketitle

\begin{abstract}  
We study convergence rates of the generalized conditional gradient (GCG) method applied to fully discretized Mean Field Games (MFG) systems.
While explicit convergence rates of the GCG method have been established at the continuous PDE level, a rigorous analysis that simultaneously accounts for time–space discretization and iteration errors has been missing.
\new{In this work, we consider an MFG system with a quadratic Hamiltonian including an advection term. We employ the GCG method to construct an iterative algorithm in which the HJB and FP equations are solved separately. At each iteration, we apply the Cole--Hopf transformation to linearize these equations and discretize the resulting linear PDEs using a finite difference method. We then analyze the resulting fully discrete GCG scheme.
Under suitable structural assumptions on coupling terms,} we establish discrete maximum principles and derive explicit error estimates that quantify both discretization errors and iteration errors within a unified framework.
Our estimates show how the convergence rates depend on the mesh sizes and the iteration number, and they reveal a non-uniform behavior with respect to the iteration.
Moreover, we prove that higher convergence rates can be achieved under additional regularity assumptions on the solution.
Numerical experiments are presented to illustrate the theoretical results and to confirm the predicted convergence behavior.
\end{abstract}

\bigskip


\bigskip

\noindent \textbf{AMS classification (2020):} 
65M06, 
35Q89, 
49N80, 

\section{Introduction}

The Mean Field Games (MFG) system provides a continuous mathematical framework for describing approximate Nash equilibria of non-cooperative games involving a very large population of interacting agents.
It was independently introduced in 2006 by Lasry and Lions \cite{ll06a,ll06b,LL07}, and by Huang et al.\cite{H06}, motivated by different perspectives: economics in the former case and control of multi-agent systems in the latter.
More precisely, the MFG system is formulated as a nonlinear system of partial differential equations (PDE) consisting of the Hamilton--Jacobi--Bellman (HJB) equation, which determines the optimal control $\nabla u$, and the Fokker--Planck (FP) equation, which determines the evolution of the agents’ density distribution $m$. The MFG system has found applications across a wide range of disciplines, including control of autonomous vehicle swarms and crowd dynamics, as well as mathematical biology, engineering, economics, and machine learning.
Despite their mathematical complexity, the MFG system has made significant contributions to decision-making under uncertainty and the control of large-scale dynamical systems; see 
\cite{MR3752669,MR3753660,lauriere2024,10508221} for more detail. 
Consequently, the development of numerical methods that can compute the MFG solutions both accurately and within realistic computational times is of great practical importance. 
In fact, many studies have been reported from a numerical analysis viewpoint, covering finite difference methods \cite{ach13, A13, A12, AC10, MR4621911, G12, I23}, finite element methods \cite{and17, os25}, deep neural networks (DNNs) \cite{MR4236167}, and so on.

However, introducing a discretization alone is not sufficient for obtaining a practical numerical solver for the MFG system. 
The MFG system is an initial–terminal value problem composed of nonlinear PDEs evolving forward and backward in time, or equivalently, a $(d+1)$-dimensional space–time boundary value problem, where $d$ denotes the spatial dimension. As a result, unlike standard initial value problems, the MFG system cannot be solved by simple forward time marching; instead, the entire space–time domain must be treated simultaneously. Even in two spatial dimensions, discretization leads to a finite-dimensional system with an enormous number of degrees of freedom, making its solution computationally demanding and typically requiring large-scale computing resources and long runtimes. This difficulty implies that solving the MFG system within realistic time constraints is far from trivial.

To address these challenges, decoupling-based iterative methods have proven effective. The basic idea is to obtain the MFG solution by alternately and iteratively solving the HJB and FP equations.
Representative examples of such approaches include monotone iterations \cite{G12} and policy iteration methods \cite{C21,CT22}.
In this paper, we focus the \emph{Generalized Conditional Gradient} (GCG) method \cite{LP23}.
In particular, Lavigne and Pfeiffer \cite{LP23} were the first to derive explicit convergence rates for the GCG method under general Hamiltonians.
Their analysis focused on spatially global coupling terms with the potentials (as well as price-based coupling effects, which are not considered in the present work). In our previous work \cite{NS26}, we extended their results to the case of locally coupled interactions (still having the potentials), even though under the restriction to a prototypical Hamiltonian that will be specified later.
However, both their analysis and ours are conducted at the PDE level.
Although numerical experiments are reported, a rigorous theoretical analysis of the discretization itself is not provided. In fact, one of the authors studied in \cite{I23} the convergence of the fictitious-play iteration applied to fully discretized MFG systems by the finite difference schemes. For a certain class of MFG systems, convergence was proved while taking into account both the discretization and the iterative procedure. This study was motivated by the analysis of monotone iteration methods for the MFG system developed by \cite{G12}. However, explicit convergence rates were not addressed. 
In particular, to the best of our knowledge, there is currently no work that \emph{simultaneously} derives explicit error estimates for both the \emph{discretization error} and the \emph{iteration error} when applying the GCG method to fully discretized MFG systems relevant for practical computations.

The main objective of this paper is to address this open problem by discretizing the MFG system using finite difference methods and providing a unified analysis of both sources of error. We now introduce the class of MFGs considered in this work.
Letting $Q := (0, T) \times \T^d$ with the $d$-dimensional torus $\mathbb{T}^d := (\mathbb{R}/\mathbb{Z})^d$ and $T > 0$, we consider \begin{subequations}
\label{MFG_eq}
\begin{alignat}{2}
-\del_t u - \nu \Delta{u} + H(t, x, \nabla u) &= f(t,x,m) &\quad& \mbox{in } Q, \label{MFG_eq1}\\ 
u(T,x) &= g(x) && \mbox{on } \T^d, \label{MFG_eq2}\\
\del_t m - \nu \Delta{m} - \nabla \cdot \left( m \nabla_p H(t, x, \nabla u)\right) &= 0 && \mbox{in } Q, \label{MFG_eq3}\\ 
m(0,x) &= m_0(x) && \mbox{on } \T^d,\label{MFG_eq4}
\end{alignat}
\end{subequations}
where $\nu > 0$ is a given constant. The unknown functions are $u:Q\to \mathbb{R}$ and $m:Q\to \mathbb{R}$, which we call the value and density functions, respectively. We are given $g : \T^d \to \R$, 
$m_0 : \T^d \to \R$, and 
$f: Q \times \mathcal{D}_0(\T^d) \to \R$, and they are called the terminal condition, initial condition, and coupling term, respectively. The space $\mathcal{D}_0(\T^d)$ denotes a set of density distributions on $\T^d$ which will be defined in \eqref{eq:density} below. 
We interpret the right-hand side of \eqref{MFG_eq1} as \(f(t,\cdot,m(t,\cdot))\), and, when no confusion arises, write it simply as \(f(t,x,m)\) or \(f(t,\cdot,m)\). The function $H : Q \times \R^d \to \R$ is called the Hamiltonian. 

The GCG method reads as follows: starting with the initial guess $\bar{m}_0$, we generate the sequences ${u}_{k}$, ${m}_{k}$, $\gamma_k$ and $\bar{m}_{k+1}$ for $k=0,1,\ldots$ by 
\begin{subequations} 
\label{GCG}
\begin{alignat}{3}
\gamma_k(t,\cdot) &= f(t,\cdot,\bar{m}_k) &\quad& \mbox{in } Q, \label{GCGg}\\ 
 -\del_t u_k - \nu \Delta{u}_k + H(t, \cdot, \nabla{u}_k) &= \gamma_k(t,\cdot) &\quad& \mbox{in } Q, \label{GCGa}\\ 
 u_k(T, \cdot) &= g && \mbox{on } \T^d,\label{GCGb} \\
\del_t m_k - \nu \Delta{m}_k - \nabla \cdot (m_k \nabla_p H(t, \cdot, \nabla u_k)) &= 0 && \mbox{in } Q, \label{GCGc}\\ 
m_k(0, \cdot) &= m_0 && \mbox{on } \T^d, \label{GCGd}
\end{alignat}
and $\bar{m}_{k+1}$ is updated by 
\begin{equation}
\bar{m}_{k+1}=(1-\delta_k) \bar{m}_k+\delta_k m_k\quad \mbox{in }Q,
\label{GCGf}
\end{equation}
\end{subequations}
where $\delta_k\in (0,1)$ is a suitably chosen parameter. For example, if $\delta_k$ is chosen as $\delta_k = 1/(k+1)$, the updating rule is rewritten as 
\begin{equation*}
\bar{m}_{k+1} = \frac{1}{k+1} \sum_{j=0}^{k} m_j.
\end{equation*}
In this case, \eqref{GCG} is referred to as the \emph{fictitious-play} method studied in \cite{CH17} and, therefore, the GCG method is interpreted as a generalization of the fictitious-play method. Moreover, several alternative choices of $\delta_k$ may be employed to accelerate convergence. In particular, an \emph{adaptive step-size} selection based on the variational structure of the MFG system can be effective.
While this strategy can indeed reduce the number of iterations, the computational cost per iteration is not negligible. 
(See Section \ref{sec:remarks}, Item 5 for more detail about adaptive step-size selection methods.) 
On the other hand, by generalizing the fictitious-play iteration, one may adopt a simpler choice of the form 
\begin{equation}
\label{eq:predefine}
\delta_k = \frac{k_2}{k+k_1}\qquad 
(\mbox{$1 \le k_2 \le k_1$ are given constants}). 
\end{equation}
We refer to this approach as a \emph{predefined step-size} selection. With an appropriate choice of the parameters $k_1$ and $k_2$, this strategy is sufficiently efficient and has the advantage of a lower computational cost per iteration.
In this paper, we therefore always assume the predefined step-size selection \eqref{eq:predefine} for $\delta_k$. Moreover, our error analysis is carried out with the aim of providing guidance for a strategic choice of the parameters $k_1$ and $k_2$. 

Throughout this paper, we assume that $H$ is given as
\begin{equation}
  \label{eq:quadH}
H(t, x, p) = \frac{1}{2} |p|^2 - h \cdot p,
\end{equation}
where $h : Q \to \R^d$ is a given vector-valued function. ($|\cdot|$ and $\cdot$ denote the Euclid norm and inner-product in $\R^d$, respectively.) 
Under this assumption, the problem setting considered in this paper cannot be regarded as fully general. Nevertheless, from the viewpoint of control engineering, this choice of $H$ corresponds to modeling microscopic system dynamics with advection effects $h$, and is recognized as an important and relevant class in control theory; see \cite[\S 2]{I23}, \cite[Chapter 6]{Nijmeijer1990} and \cite{Bagdasaryan2019} for example. 
As discussed in detail in our previous work \cite{NS26}, to establish the convergence of the GCG method, one needs a uniform estimate for $\|\nabla u_k\|_{L^\infty(Q)}$ that is independent of $k$. In particular, in the case of local coupling terms, which are the main focus of the present paper, the strong convexity of the corresponding Lagrangian, in addition to the regularity of $H$, plays an essential role in deriving such an estimate. Consequently, the class of $H$ that can be treated is restricted to \eqref{eq:quadH}. 

Furthermore, considering \eqref{eq:quadH} offers the following advantages.
That is, we follow the approach of \cite{I23} and reformulate the GCG method into a more convenient form prior to discretization, instead of discretizing it directly. Thus, we apply the \emph{Cole--Hopf transformation} 
\begin{equation}
\label{eq:ch}
\phi_k = \exp\left(\frac{-u_k}{2 \nu}\right)
\end{equation}
with $m_k = \phi_k \psi_k$. Then, 
in the smooth framework, the GCG method \eqref{GCG} admits the following equivalent formulation: starting with the initial guess $\bar{m}_0$, we generate the sequences ${\phi}_{k}$, ${\psi}_{k}$, $\gamma_k$ and $\bar{m}_{k+1}$ for $k=0,1,\ldots$ by 
\begin{subequations} 
\label{CH_MFG_eq}
\begin{alignat}{2}
\gamma_k(t,\cdot) &= f(t,\cdot,\bar{m}_k) &\quad& \mbox{in } Q, \label{CH_MFG_eq-0}\\ 
\del_t \phi_k + \nu \Delta{\phi_k} + h\cdot \nabla \phi_k &= \frac{1}{2 \nu} \gamma_k \phi_k &\quad& \mbox{in } Q, \label{CH_MFG_eq-1}\\
\phi_k(T,\cdot) &= \exp{\left(-\frac{g}{2 \nu}\right)}&\quad& \mbox{on }\T^d ,\label{CH_MFG_eq-2}\\
\del_t \psi_k - \nu \Delta{\psi_k} + \nabla \cdot (h \psi_k) &= -\frac{1}{2 \nu} \gamma_k \psi_k && \mbox{in } Q, \label{CH_MFG_eq-3}\\
 \psi_k(0,\cdot) &= \frac{m_0}{\phi_k(0, \cdot)} && \mbox{on } \T^d,\label{CH_MFG_eq-4}\\
m_k &= \phi_k \psi_k, &\quad&\mbox{in }Q, \label{CH_MFG_eq-5}
\end{alignat}
and $\bar{m}_{k+1}$ is updated by 
\begin{equation}
\bar{m}_{k+1}=(1-\delta_k) \bar{m}_k+\delta_k m_k\quad \mbox{in }Q.
\label{CH_MFG_eq-6}
\end{equation}
\end{subequations}
In what follows, we call \eqref{CH_MFG_eq} the \emph{CH-GCG method}. 
In Section \ref{sec:main}, we introduce a finite difference approximation for \eqref{CH_MFG_eq}, referred to as the \emph{discrete GCG scheme}. The term ``scheme'' is used instead of ``method'' to emphasize that it consists of a finite system of unknowns. 

Of course, it would be highly interesting to consider general Hamiltonians and discretize the MFG system directly in the original variables, without relying on the Cole–Hopf transformation. However, owing to various technical difficulties, we do not pursue this direction in the present paper. For further details, we refer to Items 2 and 3 in Section \ref{sec:remarks}.

As for the coupling terms, we are mainly interested in the local coupling terms including
\begin{equation}
\label{eq:ex-f}
f(t, x, m) = C_1\min\{m(t,x)^\alpha,\beta\}
\end{equation}
where $C_1>0$, $\alpha\ge 1$ and $\beta\ge 1$ are given constants. Here, $m$ is assumed to be non-negative and $\beta$ can be taken sufficiently large. Here, the constant $\beta>0$ is introduced to ensure the uniform boundedness of the coupling term $f$, which is later assumed in \textup{(f-B)}.
This is only a prototypical example, not an exhaustive list.
The precise assumptions imposed on $f$ considered in our work will be stated in \textup{(f-B)}, \textup{(f-L)}, \textup{(f-M)}, and \textup{(f-P)} below. Such local coupling terms are of significant practical relevance. As a concrete example, Bonnemain et al. \cite{B23} proposed an MFG model with $f(t,x,m)=m(t,x)$ to describe the response of a crowd to an intruder moving through a static group.
They demonstrated that MFG successfully capture rational lateral avoidance behavior of the crowd, which cannot be reproduced by traditional mathematical models.
Furthermore, in many studies on the MFG system, numerical experiments are often performed using local coupling terms; see \cite{ach13, A12, A13, AC10, C21, G12, I23, MR4596353, os25} for example. 
Nevertheless, our analysis can also be applied to globally coupling terms.
Additional explanations on this point are given in Section \ref{sec:gct}.

Here, we briefly comment on the distinction between potential and non-potential frameworks and clarify the role of the potential assumption in the present analysis. As mentioned above, the numerical scheme studied in this paper is based on the GCG method. Accordingly, we assume both the classical Lasry--Lions monotonicity condition $\textup{(f-M)}$ and the existence of a potential for the coupling term $\textup{(f-P)}$.

The potential structure is intrinsic to the variational formulation and analysis of the method. It allows the MFG system to be formulated as a minimization problem and the iteration to be interpreted and analyzed as a GCG method for the associated objective functional; see Section \ref{sec:preliminary}. Thus, the potential assumption cannot simply be removed while retaining the present variational GCG formulation and analysis. Treating non-potential couplings would instead require an analytical framework that does not rely on the potential functional.

Such non-potential methods have been studied using different analytical frameworks. At the PDE level, examples for second-order MFGs include \cite{C21,CT22}. At the discrete level, such methods have been considered for first-order MFGs in \cite{CS14,CC26} and for finite-state MFGs in \cite{HS19}, while \cite{G12} treats both PDE-level and discrete-level iterations for second-order MFGs. Depending on the setting, these approaches are based on policy iteration, monotone-operator techniques, fictitious play, or monotonicity arguments, rather than on the minimization of a potential functional.

The potential structure is also essential for the quantitative PDE-level convergence estimate used in this paper. As explained at the beginning of Section \ref{sec:proof}, deriving an explicit iteration-error estimate requires a convergence bound for the PDE-level iteration with an explicit dependence on the iteration number. To the best of our knowledge, the estimates directly applicable to the general setting considered here are those obtained in the GCG analysis introduced in \cite{LP23} and subsequently extended in \cite{NS26}; in particular, we use the result of \cite{NS26}. These estimates are formulated through an optimality gap defined in terms of the potential functional.

The non-potential analysis in \cite{CT22} also gives a quantitative PDE-level estimate, but under a restrictive smallness condition on the coupling term. For example, for $f(t,x,m)=C_1\min\{m(t,x)^\alpha, \beta\}$ with $\alpha \ge 1$ and $\beta>0$, this condition requires $C_1$ to be sufficiently small, whereas a potential exists for arbitrary $C_1$. Therefore, this result is not applicable to the general setting considered here. Apart from results under such additional restrictive assumptions, no quantitative PDE-level convergence estimate for a non-potential iterative framework appears to be currently available in a form suitable for the present error analysis.

At this stage, we briefly summarize our main results. 
As will be stated in Theorem \ref{thm:DMP}, the solutions of the discrete GCG scheme inhibit the discrete maximum principle that plays an important role in error analysis. Our error estimates are, for example, 
\begin{equation*}
\left[\sum_{n=0}^{N_t-1} \left(\max_{i} |\bar{M}_{n,i}^k - \bar{m}(t_n,x_i)|\right)^2 \Delta t\right]^{\frac{1}{2}} \le C_{\bar{m}}^{(1)}(k) \left[(\Delta t)^{\frac{\eta}{2}} + (\Delta x)^{\eta}\right] + \frac{C_{\bar{m}}^{(2)}}{(k+k_1)^{\frac{s}{r}}}
\end{equation*}
Here, \(\bar{M}_{n,i}^k\) approximates \(\bar{m}(t_n,x_i)\) at the grid point \((t_n,x_i)\), with grid sizes \(\Delta t\) and \(\Delta x\), and $N_t$ denotes the number of subdivisions of the interval $[0,T]$. 
The quantities \(\eta \in (0,1)\), \(s>0\), \(r>d\), and \(C_{\bar{m}}^{(2)}>0\) are constants, while \(C_{\bar{m}}^{(1)}(k)>0\) depends on \(k\).
We investigate these error estimates within the problem setting in which the MFG system \eqref{MFG_eq} admits a classical solution, as established in \cite{NS26}; see Proposition \ref{B21_thm1}. 
As a consequence, we succeed in proving the above error estimates without imposing any additional regularity assumptions on the solution (see Theorem \ref{I23_thm3.5}). Moreover, we show that, if higher regularity can be assumed, the convergence rate $(\Delta t)^{\frac{\eta}{2}} + (\Delta x)^{\eta}$ can be improved to $\Delta t + \Delta x$ or $\Delta t + (\Delta x)^2$ (see Theorems \ref{I23_thm3.5a} and \ref{I23_thm3.5b}). 
However, at this point, we must make the following important remark. 
The growth of $C_{\bar{m}}^{(1)}(k)$ is 
bounded from above by a polynomial function as $k\to\infty$ (see Theorem \ref{I23_thm3.5}). Furthermore, in many practically relevant settings, including the one considered in Section \ref{sec:numerical_experiment}, $C_{\bar{m}}^{(1)}(k)$ does in fact diverge logarithmically; see Remark \ref{rem:log} for mote details.
Thus, the convergence property above is not uniform with respect to the iteration number $k$. Nevertheless, the above error estimates still allow us to obtain convergence results with respect to both the discretization and the iteration (see Remark \ref{rem:convergence}). 

We conclude this introduction by outlining the structure of the paper.

In Section~\ref{sec:settings}, we present the basic problem setting and review several results established in \cite{NS26}, including the existence of classical solutions to the MFG system (Proposition~\ref{B21_thm1}) and uniform boundedness results for the sequences generated by the GCG method (Proposition~\ref{prop:initial} and Proposition~\ref{LP23_prop23}).
In Section~\ref{sec:main}, we describe the details of the discretization, in particular the discrete GCG scheme, and state our main results, namely Theorems~\ref{thm:DMP}, \ref{I23_thm3.5}, \ref{I23_thm3.5a}, and \ref{I23_thm3.5b}.
Section~\ref{sec:preliminary} is devoted to preparatory material for the proofs of the main results.
In particular, we recall convergence results for the GCG method previously studied in \cite{NS26}, and we further establish new regularity results for solutions of the CH–GCG scheme.
The proofs of the main theorems are then presented in Section~\ref{sec:proof}.
In Section~\ref{sec:numerical_experiment}, we report the results of numerical experiments.
In Section~\ref{sec:gct}, we derive error estimates corresponding to our main results under the assumptions on the coupling term considered in \cite{LP23}.
Finally, in Section~\ref{sec:remarks}, we summarize our findings and provide several concluding remarks.

\subsection*{Notation}

We here collect function spaces and norms used in this paper. 

For a smooth function $u:Q\to \mathbb{R}$, we write $\del_i u= \del_{x_i}u$ and 
$\del_{i}\del_j u= \del_{x_i}\del_{x_j}u$ and so on. For a smooth vector-valued function $v=(v_i):Q\to \mathbb{R}^d$, we write $D_x v = (\del_j v_i)_{1 \le i, j \le d} \in \R^{d \times d}$.  

Let $X$ denote $\T^d$, $Q$ or $(0,T)$. 
For any $k \in \Z_{\ge 0}$, $\mathcal{C}^k(X; \R^d)$ denotes the space of functions of class $\mathcal{C}^k$ on $X$ valued in $\R^d$. We simply write $\mathcal{C} := \mathcal{C}^0$ and $\mathcal{C}^k(X) := \mathcal{C}^k(X;\R)$.

For any $\alpha \in (0, 1)$, $\mathcal{C}^{k + \alpha}(X;\R^d)$ denotes the H\"{o}lder space of functions whose derivatives up to order $k$ are H\"{o}lder continuous of order $\alpha$. 
The function $u$ of $\mathcal{C}(Q)$ belongs to $\mathcal{C}^{\frac{\alpha}{2}, \alpha}(Q)$ if and only if $u(\cdot, x) \in \mathcal{C}^{\frac{\alpha}{2}}(0, T)$ and $u(t, \cdot) \in \mathcal{C}^{\alpha}(\T^d)$. 
If $\del_t u, \del_i u, \del_i \del_j u \in \mathcal{C}^{\frac{\alpha}{2}, \alpha}(Q)$ ($1 \le i, j \le d$), we write as $u\in \mathcal{C}^{1+\frac{\alpha}{2}, 2+\alpha}(Q)$. 

For any $p \in [1, \infty]$, $L^p(X)$ and $W^{k, p}(X)$ denote the standard Lebesgue and Sobolev spaces, respectively. 
We use the standard Bochner space $L^p(0,T;W^{k,p}(\mathbb{T}^d))$. 
Throughout this paper, we fix a constant $q\in\mathbb{R}$ satisfying 
\begin{equation}
q > d+2.
\end{equation}
Then, set 
$$
W^{1, 2, q}(Q) := W^{1, q}(Q) \cap L^q(0, T; W^{2, q}(\T^d)),
$$
and 
$$
\|u\|_{W^{1, 2, q}(Q)} := \|u\|_{W^{1, q}(Q)} + \|u\|_{L^q(0, T; W^{2, q}(\T^d))}.
$$
The space of control inputs in the FP equation is defined as
$$
\Theta := \left\{v \in \mathcal{C}(Q; \R^d) \mid D_x v \in L^q(Q; \R^{d \times d})\right\}, 
$$
equipped with the norm
$$
\|v\|_{\Theta} := \|v\|_{L^{\infty}(Q: \R^d)} + \|D_x v\|_{L^q(Q; \R^{d \times d})},
$$
where $D_x v = (\del_j v_i)_{1 \le i, j \le d} \in \R^{d \times d}$ is the Jacobi matrix of $v = (v_1, \dots, v_d)^{\top}$.

Define the space of coupling terms and its norm as 
$$
\Gamma := L^{\infty}(0, T; W^{1, \infty}(\T^d))\cap \mathcal{C}(Q) 
,\quad 
\|\gamma\|_{\Gamma} := \|\gamma\|_{L^{\infty}(0, T; W^{1, \infty}(\T^d))}.
$$
The set of density distributions on $\T^d$ is defined as
\begin{equation}
\label{eq:density}
\mathcal{D}_0(\T^d) := \left\{m \in \mathcal{C}(\T^d) \mid m \ge 0\right\}.
\end{equation}

Finally, unless otherwise stated, the letter $C$ denotes a generic positive constant independent of the discretization parameters $\Delta t$, $\Delta x$ and iteration number $k$. Whenever we need to emphasize that the constant $C$ depends on other parameter $R$, we express it by $C(R)$.

\section{Problem settings}
\label{sec:settings}
We make the following assumptions on $H$, $f$, $g$ and $m_0$. 

\begin{description} \label{assum-L}
\item[(H)] The Hamiltonian $H$ is given by \eqref{eq:quadH} where $h \in \mathcal{C}^{1+\alpha_0}(Q; \R^d)$ with a constant $\alpha_0\in (0,1)$. 
\item[(f-B)] The coupling term $f$ is non-negative and bounded. That is, there exists a constant $C_0$ satisfying:
\label{eq:f-123}
\begin{equation}
0\le f(t, x, m) \le C_0
\label{eq:f-1}
\end{equation}
for $(t, x) \in Q$, and $m\in \mathcal{D}_0(\T^d)$. 

\item[(f-L)] The coupling term $f$ is Lipschitz continuous. That is, there exist constants $C_0$, $L_f > 0$ and $\alpha_0 \in (0, 1)$ satisfying: 
\begin{subequations}
\begin{equation}
|f(t_2, x, m_2) - f(t_1, x, m_1)| \le C_0 |t_2-t_1|^{\alpha_0} + L_f \|m_2 - m_1\|_{L^{\infty}(\T^d)}
\label{eq:f-L1a}
\end{equation}
for $(t_1, x),(t_2, x) \in Q$, and $m_1,m_2 \in \mathcal{D}_0(\T^d)$. 
Moreover, if $m \in \mathcal{D}_0(\T)$ is a Lipschitz continuous function (of $x$) with the Lipschitz constant $\operatorname{Lip}_x(m)$, then $f(t,\cdot,m)$ is also a Lipschitz continuous function of $x$ with the Lipschitz constant $\operatorname{Lip}_x(m;f)$: 
\begin{equation}
\label{eq:f-L1b}
|f(t, x_2, m) - f(t, x_1, m)| \le \operatorname{Lip}_x(m;f) |x_2 - x_1|
\end{equation}
\end{subequations}
for $(t, x_1),(t, x_2) \in Q$. Moreover, if $\operatorname{Lip}_x(m)\le R$ for some $R>0$, there is a constant $C(R)$ depending on $R$ such that $\operatorname{Lip}_x(m;f)\le C(R)$.

\item[(f-M)] The coupling term $f$ is monotone with respect to the third argument. That is, 
\begin{equation}
\label{eq:f-4}
\int_{\T^d} \left[f(t, x, m_2) - f(t, x, m_1)\right][m_2(x) - m_1(x)]\, dx \ge 0
\end{equation}
for $t \in [0, T]$ and $m_1,m_2 \in \mathcal{D}_0(\T^d)$. 
\item[(f-P)] The coupling term $f$ has a potential. That is, there exists a continuous function $F : [0, T] \times \mathcal{D}_0(\T^d) \to \R$ such that
\begin{equation}
\label{eq:f-5}
F(t, m_2) - F(t, m_1) = \int_0^1 \int_{\T^d} f(t, x, sm_2 + (1-s)m_1)[m_2(x)-m_1(x)]~ dxds.
\end{equation}
\item[(TIV)] $g \in \mathcal{C}^{2 + \alpha_0}(\T^d)$ and $m_0 \in \mathcal{D}_0(\T^d) \cap \mathcal{C}^{2 + \alpha_0}(\T^d)$ with a constant $\alpha_0\in (0,1)$. 
\end{description}

\begin{rem}
\label{rem:fB}
In most works on MFG systems (see, e.g., \cite{LP23,NS26}), it is assumed that
\[
|f(t,x,m)| \le C_0,
\]
which is more general than assumption \textup{(f-B)}.
However, our assumption \textup{(f-B)} is not restrictive, since replacing \(f\) by \(f+C_0\) leaves the optimal control \(\nabla u\) and the density
\(m\) in the MFG system \eqref{MFG_eq} unchanged; see \cite[p.~4]{G12}.
\end{rem}

\begin{rem}
  \label{rem:potential}
  Under assumption \textup{(f-P)}, the MFG system \eqref{MFG_eq} is referred to as a \emph{potential MFG}.
\end{rem}

\begin{rem}
  \label{rem:assumptions_fM_fP}
  We clarify here the role of the assumptions \textup{(f-M)} and \textup{(f-P)}. These assumptions are used to ensure the convergence of the GCG method through Propositions~\ref{LP23_thm7} and~\ref{LP23_thm8-2}, and also to guarantee the uniqueness of solutions to the MFG system \eqref{MFG_eq}. Apart from these specific purposes, they are not used in the proofs of the main results of this paper.
\end{rem}

\begin{rem}
\label{rem:constants}
The constant $\alpha_0$ appearing in \textup{(H)}, \textup{(f-L)}, and \textup{(TIV)} are assumed to take the same value. 
The constant $C_0$ in \textup{(f-B)}, \textup{(f-L)} are assumed to take the same value too. 
\end{rem}

\begin{rem}
In this paper, we treat the coupling term $f$ as a mapping from $Q \times \mathcal{D}_0(\mathbb{T}^d)$ to $\mathbb{R}$. 
This is different from \cite{NS26}, where $f$ is treated as a mapping from 
$Q \times \mathcal{D}_1(\mathbb{T}^d)$ to $\mathbb{R}$. 
Here,
\begin{equation*}
\mathcal{D}_1(\mathbb{T}^d) := \left\{ m \in \mathcal{C}(\mathbb{T}^d) \mid m \ge 0,\; \int_{\mathbb{T}^d} m(x)\,dx = 1 \right\}.
\end{equation*}
This difference arises because, for approximation functions of the density $m(x,t)$, 
the total mass is not necessarily equal to $1$ in general (see Remark \ref{rem:mass_conservation} for details). 
However, in the convergence analysis presented in Section~\ref{sec:proof}, 
the constraint that the total mass be equal to $1$ is not essential and is not required. 
Therefore, treating $f$ as a mapping from $Q \times \mathcal{D}_0(\mathbb{T}^d)$ to $\mathbb{R}$ 
does not cause any inconvenience.
\end{rem}

We recall the fundamental results from \cite[Theorems 2.7, 2.10]{NS26}. We note that the regularity result (Proposition \ref{prop:initial}) is standard. A corresponding statement under assumptions compatible with ours can be found in the proof of \cite[Theorem 2.10]{NS26}, based on the classical parabolic regularity theory of \cite{lsu68}.

\begin{prop}
\label{B21_thm1}
Suppose that 
\textup{(H)}, 
\textup{(f-B)}, \textup{(f-L)}, and \textup{(TIV)} are satisfied. 
Then, the MFG system \eqref{MFG_eq} admits a unique solution $(\bar{u},\bar{m})\in W^{1,2,q}(Q)\times W^{1,2,q}(Q)$. Moreover, $\bar{\gamma}=f(\cdot,\cdot,\bar{m})\in\Gamma$. Furthermore, they are classical solutions in the sense that 
\[
(\bar{m}, \bar{u}, \bar{\gamma}) \in \mathcal{C}^{1+\frac{\alpha}{2}, 2 + \alpha}(Q) \times \mathcal{C}^{1+\frac{\alpha}{2}, 2+\alpha}(Q) \times \mathcal{C}^{\alpha}(Q)
\]
hold with some $\alpha\in (0,1)$. 
\end{prop}

\begin{prop}
  \label{prop:initial}
For $\alpha_0\in (0,1) $ appearing in \textup{(TIV)}, let 
$m_0\in \mathcal{D}_0(\T^d) \cap \mathcal{C}^{2 + \alpha_0}(\T^d)$ and 
$v\in \mathcal{C}^{\alpha_0/2,1+\alpha_0}(Q;\mathbb{R}^d)$ be given. 
Then, there exist a constant $\beta\in (0,1)$ depending on $\alpha_0$ and a unique solution $\bar{m}_0\in \mathcal{C}^{1+\frac{\beta}{2},2+\beta}(Q)$ of the initial value problem $\del_t \bar{m}_0 - \nu\Delta{\bar{m}_0}+\nabla\cdot (v\bar{m}_0) = 0$ in $Q$ and $\bar{m}_0(0) = m_0$ on $\T^d$. Moreover, we have 
$\bar{m}_0\in W^{1,2,q}(Q)$ and $\bar{m}_0(t,\cdot)\in\mathcal{D}_0(\T^d)$ for $t\in [0,T]$. 
\end{prop}

\begin{prop}
\label{LP23_prop23}
Suppose that 
\textup{(H)}, 
\textup{(f-B)}, 
\textup{(f-L)}, 
\textup{(f-M)}, and \textup{(TIV)} are satisfied. 
As the initial guess $\bar{m}_0$, we take the function specified in Proposition~\ref{prop:initial}. 

\begin{itemize}
  \item[\textup{(i)}] The GCG method \eqref{GCG} generates sequences 
$u_k\in W^{1,2,q}(Q)$, $m_k\in W^{1,2,q}(Q)$ $\gamma_k \in \Gamma$, and $\bar{m}_{k+1}\in W^{1,2,q}(Q)$ for $k=0,1,2,\ldots$. 

\item[\textup{(ii)}] These sequences are uniformly bounded. That is, there exists a constant $C > 0$ such that, for $k=0,1,2,\ldots$, 
\[
\|u_k\|_{W^{1, 2, q}(Q)}\le C,\quad 
\|m_k\|_{W^{1, 2, q}(Q)} \le C, \quad 
\|\gamma_k\|_{\Gamma} \le C,\quad 
\|\bar{m}_{k}\|_{W^{1, 2, q}(Q)}\le C.
\]
\end{itemize}

\end{prop}


\begin{rem}
We briefly compare Proposition 2.6 with existing regularity results for MFG systems with local couplings. Related results can be found, for example, in \cite{CP20, G16}. However, these results cannot be directly used to obtain Proposition 2.6 under the assumptions considered in the present paper. More specifically, Theorem 11 in \cite{CP20} assumes that $\nabla_p H$ is uniformly bounded and therefore does not cover the quadratic Hamiltonian considered here, whereas Theorem 12 in the same reference establishes only the existence of weak solutions under more general growth conditions. The results in \cite{G16} that are most closely related to Proposition 2.6 also require stronger structural and regularity assumptions (see also \cite{NS26}).

\end{rem}

\section{The discrete GCG scheme and main results}
\label{sec:main}

The time interval $[0,T]$ is discretized by $t_n = n\Delta t$, $n=0,1,\dots,N_t$, where $\Delta t = T/N_t$ and $N_t\in\mathbb{Z}_{>0}$. 
The spatial domain $\mathbb{T}^d$ is discretized using a uniform grid $x_i = i\Delta x$, where $i=(i_1,\dots,i_d)$ is a multi-index with $i_\ell=0,1,\dots,N_x$ for $\ell=1,\dots,d$, and $\Delta x = 1/N_x$ with $N_x\in\mathbb{Z}_{>0}$. 
Periodic boundary conditions are imposed in each spatial direction, that is, $i_\ell \equiv 0 \ (\mathrm{mod}\ N_x)$ for $\ell=1,\dots,d$. 
We set 
\begin{align*}
\Lambda_x&:=\{i=(i_1,\dots,i_d) \in \mathbb{Z}_{>0}^d\mid 
i_\ell \equiv 0 \ (\mathrm{mod}\ N_x),\ \ell=1,\dots,d\},\\
\Lambda&:=\{(n,i)\mid n=0,1,\ldots,N_t,~i\in\Lambda_x\}.
\end{align*}
For the advection term $h(t,x)$, we introduce grid functions $h_{n, i}$ and $h_{n, i}^{\pm}$ defined in $\Lambda$ as 
\begin{equation}
\label{eq:dh}
h_{n, i} := h(t_n, x_i)\in\mathbb{R}^d,\qquad 
h_{n, i}^{\pm} := \max{\{\pm h_{n, i}, 0\}}\in\mathbb{R}_{\ge 0}^d.
\end{equation}
We write 
$h_{n, i} =(h^{1}_{n, i},\ldots,h^{d}_{n, i})$ and $h_{n, i} =(h^{1, \pm}_{n, i},\ldots,h^{d, \pm}_{n, i})$. 

For a grid function $\phi_{n,i}$ defined in $\Lambda$, we introduce the following difference quotients: 
\begin{align*}
D_1^1(\phi, n, i) &:= \sum_{\ell=1}^d \left(h^{\ell, +}_{n, i} \frac{\phi_{n, i[\ell^+]} - \phi_{n, i}}{\Delta x} - h^{\ell, -}_{n, i} \frac{\phi_{n, i} - \phi_{n, i[\ell^-]}}{\Delta x}\right),\\
D_1^2(\phi, n, i) &:= \sum_{\ell=1}^d \bigg(\frac{1-\sgn{h^{\ell}_{n, i}}}{2} \frac{h^{\ell}_{n, i[\ell^+]} \phi_{n, i[\ell^+]} - h^{\ell}_{n, i} \phi_{n, i}}{\Delta x} \\
& \qqquad + \frac{1+\sgn{h^{\ell}_{n, i}}}{2} \frac{h^{\ell}_{n, i} \phi_{n, i} - h^{\ell}_{n, i[\ell^-]} \phi_{n, i[\ell^-]}}{\Delta x}\bigg),\\
D_2(\phi, n, i) &:= \sum_{\ell=1}^d \frac{\phi_{n, i[\ell^+]} - 2 \phi_{n, i} + \phi_{n, i[\ell^-]}}{(\Delta x)^2},
\end{align*}
where 
$i[\ell^\pm] := (i_1, \dots, i_{\ell-1}, i_{\ell}\pm 1, i_{\ell+1}, \dots, i_d)$. 

\begin{rem}
\label{rem:ap}
Let $\phi_{n,i}$ denote the value of a smooth function $\phi(t,x)$ at $(t_n,x_i)$. 
Then $D_1^1(\phi,n,i)$ denotes an upwind finite difference approximation of $h\cdot\nabla\phi$, 
$D_1^2(\phi,n,i)$ an upwind finite difference approximation of $\nabla\cdot(h\phi)$, 
and $D_2(\phi,n,i)$ a finite difference approximation of $\Delta\phi$.
\end{rem}

We introduce the grid functions
\begin{subequations}
\label{eq:sequence}
\begin{gather}
\Phi^k=(\Phi^k_{n,i})_{n,i}, \qquad \Psi^k=(\Psi^k_{n,i})_{n,i},\\
\bar{M}^k=(\bar{M}^k_{n,i})_{n,i}, \qquad M^k=(M^k_{n,i})_{n,i}, \qquad \Gamma^k=(\Gamma^k_{n,i})_{n,i},
\end{gather}
\end{subequations}
which are the discrete counterparts of $\phi_k$, $\psi_k$, $\bar{m}_k$, $m_k$, and $\gamma_k$, respectively. The discrete GCG scheme for computing these grid functions is defined as follows.
For the initial guess $\bar{m}_0$, set $\bar{M}^0_{n, i} := \bar{m}_0(t_n, x_i)$ for $(n,i)\in \Lambda$. We generate the sequences \eqref{eq:sequence} for $k=0,1,\ldots$ by 
 \begin{subequations}
 \label{eq:dGCG}
\begin{alignat}{2}
\Gamma^k_{n, i} &= f(t_n, x_i, \bar{M}^k_n),\label{eq:dGCG0}\\
\frac{\Phi^k_{n, i} - \Phi^k_{n-1, i}}{\Delta t} + \nu D_2(\Phi^k, n, i) + D_1^1(\Phi^k, n, i) &= \frac{1}{2 \nu} \Gamma^k_{n, i} \Phi^k_{n-1, i},\label{eq:dGCG2}\\
\Phi^k_{N_t, i} &= \exp{\left(-\frac{g(x_i)}{2 \nu}\right)},\label{eq:dGCG2b} \\
\frac{\Psi^k_{n+1, i} - \Psi^k_{n, i}}{\Delta t} - \nu D_2(\Psi^k, n, i) + D_1^2(\Psi^k, n, i) &= -\frac{1}{2 \nu} \Gamma^k_{n, i} \Psi^k_{n+1, i},\label{eq:dGCG3}\\
\Psi^k_{0, i} &= \frac{m_0(0,x_i)}{\Phi^k_{0, i}}, \label{eq:dGCG3b}\\
{M}^k_{n, i} &= \Phi^k_{n, i} \Psi^k_{n,i}\label{eq:dGCG4}
\end{alignat}
and 
\begin{equation}
\label{eq:dGCG5}
\bar{M}^{k+1}_{n, i}= (1 - \delta_k) \bar{M}^k_{n, i} + \delta_k M^k_{n, i},
\end{equation}
where $\delta_k \in (0, 1]$ and the function $\bar{M}^k_n : \T^d \to \R$ is defined by 
\begin{equation}
\label{eq:dGCG1}
\bar{M}^k_n(x) := \bar{M}^k_{n, i} \quad \left(x \in \prod_{\ell=1}^d \left[x_{i_{\ell}} - \frac{\Delta x}{2}, x_{i_{\ell}} + \frac{\Delta x}{2}\right) \right).
\end{equation}
\end{subequations}

\begin{rem}
\label{rem:rem-f}
Since the third argument of $f$ is a function of $\mathcal{D}_0(\T^d)$, we introduce $\bar{M}^k_n$ in the scheme for notational convenience. However, this variable is not needed in the actual implementation.
\end{rem}

\begin{rem}
  \label{rem:DMP}
The backward Euler method is applied in \eqref{eq:dGCG2}, while the forward Euler method is used in \eqref{eq:dGCG3}. Instead of applying the Euler methods directly to~\eqref{CH_MFG_eq-1}, we shift $\Phi^k_{n,i}$ backward in time by one step. An analogous technique is applied for \eqref{CH_MFG_eq-3}. 
As observed in \cite{I23}, this strategy, combined with upwind approximations, leads to a discrete maximum principle (DMP) for $\Phi^k_{n,i}$ and $\Psi^k_{n,i}$. The DMP plays an important role in the error analysis of this paper. As will be stated in Theorem \ref{thm:DMP}, under this time-shifted treatment of the reaction term, the CFL condition is \eqref{eq:cfl} described below.
On the other hand, if the reaction term were discretized in the usual way, using $\Phi^k_{n,i}$ instead of $\Phi^k_{n-1,i}$, the positivity argument would require the stronger condition
$$
\|h\|_{L^\infty(Q;\mathbb{R}^d)}\frac{\Delta t}{\Delta x}+2d\nu\frac{\Delta t}{(\Delta x)^2}+\frac{1}{2\nu}\|f\|_{L^\infty(Q)}\Delta t\le 1.
$$
Thus, the backward time shift avoids the additional restriction caused by the reaction term. The additional term $\frac{1}{2\nu}\|f\|_{L^\infty(Q)}\Delta t$ can impose a severe restriction on $\Delta t$, especially when $\nu>0$ is small or $\|f\|_{L^\infty(Q)}$ is large. Since the present iterative method requires repeatedly solving the discrete PDEs, avoiding this additional restriction is also useful from a computational point of view. Therefore, we use the time-shifted treatment of the reaction term together with upwind approximations.
\end{rem}

\begin{rem}
\label{rem:mass_conservation}
The Fokker--Planck equation satisfied by $m_k$ in \eqref{GCGc} has the usual mass conservation property. After the Cole--Hopf transformation, however, the system \eqref{CH_MFG_eq} is no longer written in a conservative form for the transformed variables $\phi_k$ and $\psi_k$. At the continuous level, this does not cause any problem: the reconstructed density $m_k=\phi_k\psi_k$ still satisfies the Fokker--Planck equation and hence preserves mass. The situation is different after discretization. In the discrete scheme \eqref{eq:dGCG}, the approximations $\Phi^k$ and $\Psi^k$ are obtained by applying standard finite difference discretizations to the Cole--Hopf transformed equations. These discrete equations do not automatically inherit the mass conservation structure of the original Fokker--Planck equation. Consequently, the reconstructed numerical density $M^k=\Phi^k\Psi^k$ does not necessarily satisfy a discrete mass conservation law. This is why the total mass of $\bar{M}^k$ is not exactly equal to one in general.
\end{rem}

\begin{rem}
For simplicity, we restrict ourselves to uniform meshes throughout this paper, although non-uniform meshes with different mesh sizes in each coordinate direction can also be treated; see~\cite{I23}.
\end{rem}

We are now ready to state our main results. The first one is the well-posedness and DMP. 

\begin{theorem}
\label{thm:DMP}
Let $h\in\mathcal{C}(Q;\mathbb{R}^d)$, $g\in\mathcal{C}(\mathbb{T}^d)$, and $m_0\in\mathcal{D}_0(\T^d)$. Suppose that \textup{(H)}, and \textup{(f-B)} are satisfied. 
Let us give an initial guess $\bar{m}_0\in\mathcal{C}(Q)$ with $\bar{m}_0\ge 0,\not\equiv 0$ for the iteration. 
Moreover, assume 
\begin{equation}
\|h\|_{L^{\infty}(Q; \R^d)} \frac{\Delta t}{\Delta x} + 2 d \nu \frac{\Delta t}{(\Delta x)^2} \le 1 .
\label{eq:cfl}
\end{equation}
Then, the discrete GCG scheme \eqref{eq:dGCG} is well-defined. That is, there exists a unique solution 
$$
(\bar{M}_{n,i}^k, {M}_{n,i}^k, \Phi_{n,i}^k, \Psi_{n,i}^k)
$$
of \eqref{eq:dGCG} for $(n,i)\in\Lambda$ and $k=0,1,\ldots$. Moreover, 
there exist positive constants $\Phi_{\max}, \Phi_{\min}$, and $\Psi_{\max}$, which are independent of the iteration number $k$, such that
\begin{equation}
  \label{eq:DMP}
\Phi_{\min} \le \Phi^k_{n, i} \le \Phi_{\max}, \qquad 0 < \Psi^k_{n, i} \le \Psi_{\max} 
\end{equation}
for $(n,i)\in\Lambda$ and $k=0,1,\ldots$. 
(This theorem holds for any choice of step sizes $\delta_k$.)
\end{theorem}

\begin{proof}
The equations \eqref{eq:dGCG2} and \eqref{eq:dGCG3} are linear for $\Phi_{n-1,i}^k$ and $\Psi_{n+1,i}^k$, respectively. 
In particular, the coefficient matrices of these systems are diagonal matrices with (strictly) positive entries. 
Hence, for each choice of the step-size $\delta_k$, a solution exists and is unique. 
The proof of \eqref{eq:DMP} follows exactly the same arguments as those in \cite[Propositions~4.3 and~4.5]{I23}. 
\end{proof}

To state error estimates, we introduce 
\begin{subequations}
\label{eq:def-IE}
\begin{align}
I(\phi,\psi) &:= \left[\sum_{n=0}^{N_t-1} \left(\max_{i\in\Lambda_x} |\phi_{n,i} - \psi(t_n,x_i)|\right)^2 \Delta t\right]^{\frac{1}{2}},\label{eq:def-I}\\
E(\phi,\psi) &:= \max_{(n,i)\in\Lambda} |\phi_{n,i} - \psi(t_n,x_i)|,\label{eq:def-E}
\end{align}
\end{subequations}
where $\phi_{n,i}$ is a grid function in $\Lambda$ and $\psi$ is a continuous function in $Q$.

Let $(\bar{m}, \bar{u})$ be the solution of the MFG system \eqref{MFG_eq} given in Proposition \ref{B21_thm1}. Moreover, let 
$(\bar{M}_{n,i}^k,{M}_{n,i}^k,\Gamma_{n, i}-k,,\Phi_{n,i}^k,\Psi_{n,i}^k)$ be the solution of \eqref{eq:dGCG} as stated in Theorem \ref{thm:DMP}. 
Define 
\begin{equation}
U^k_{n, i} := - 2 \nu \log{\Phi^k_{n, i}},
\label{eq:err-0b}
\end{equation}
which is well-defined since $\Phi_{n,i}^k>\Phi_{\min}$. This quantity approximates $\bar{u}(t_n, x_i)$.

\begin{theorem} 
\label{I23_thm3.5}
Suppose that 
\textup{(H)}, 
\textup{(f-B)}, \textup{(f-L)}, \textup{(f-M)}, \textup{(f-P)}, and \textup{(TIV)} are satisfied. 
The initial guess $\bar{m}_0$ is the function specified in Proposition~\ref{prop:initial}. 
The step-size $\delta_k$ is chosen as \eqref{eq:predefine}. 
Moreover, we choose $r \in \mathbb{R}$ and $s \in \R$ by
\begin{equation}
\label{eq:val-p}
r =
\begin{cases}
2 & (d = 1),\\
\text{any number greater than } d & (d \ge 2),
\end{cases}
\end{equation}
and
\begin{equation}
\label{eq:val-s}
s
\begin{cases}
 = k_2 & (d = 1),\\
< \min{\left\{k_2, \frac{1}{\rho}\right\}} & (d \ge 2),
\end{cases}
\end{equation}
respectively. Therein, the constant $\rho$ is defined by 
\begin{equation}
\label{eq:const-r}
\rho = 1-\frac{2}{r(r-1)} < 1.
\end{equation}
Then, there exist constants $\eta\in (0,1)$, functions $C_{\mathrm{q}}^{(1)}(k) > 0$ ($\mathrm{q} = \bar{m}, \bar{u}$) of $k$, and constants $C_{\mathrm{q}}^{(2)} > 0$ ($\mathrm{q} = \bar{m}, \bar{u}$) such that
\begin{subequations}
\label{eq:err-1}
\begin{align}
I(\bar{M}^k, \bar{m}) &\le C_{\bar{m}}^{(1)}(k) \left[(\Delta t)^{\frac{\eta}{2}} + (\Delta x)^{\eta}\right] + \frac{C_{\bar{m}}^{(2)}}{(k+k_1)^{\frac{s}{r}}}, \label{eq:err-1a}\\
E(U^k,\bar{u}) &\le C_{\bar{u}}^{(1)}(k) \left[(\Delta t)^{\frac{\eta}{2}} + (\Delta x)^{\eta}\right] + \frac{C_{\bar{u}}^{(2)}}{(k+k_1)^{\frac{s}{r}}}. \label{eq:err-1b}
\end{align}
\end{subequations}
The asymptotic growth of $C_{\mathrm{q}}^{(1)}(k)$ ($\mathrm{q} = \bar{m}, \bar{u}$) 
is polynomially bounded from above as $k\to\infty$.
\end{theorem}

The convergence rates for discretizations in \eqref{eq:err-1} can be improved if the solutions $\phi_k$ and $\psi_k$ of \eqref{CH_MFG_eq} have higher regularities uniformly in $k$. 

\begin{theorem} 
\label{I23_thm3.5a}
In addition to the assumptions of Theorem \ref{I23_thm3.5}, suppose that the solutions 
$\phi_k$ and $\psi_k$ of \eqref{CH_MFG_eq} are smooth in the sense that 
\begin{equation}
  \label{eq:C24}
  \phi_k\in \mathcal{C}^{2, 4}(Q),\quad \psi_k \in \mathcal{C}^{2, 4}(Q), 
\quad 
\|\phi_k\|_{\mathcal{C}^{2, 4}(Q)}\le C,\quad 
\|\psi_k\|_{\mathcal{C}^{2, 4}(Q)}\le C,
\end{equation}
where $C>0$ denotes a constant indepoendent of the iteration number $k$. Then, the error estimates \eqref{eq:err-1} are improved as 
\begin{subequations}
\label{eq:err-2}
\begin{align}
I(\bar{M}^k, \bar{m}) &\le C_{\bar{m}}^{(1)}(k) \left(\Delta t + \Delta x\right) + \frac{C_{\bar{m}}^{(2)}}{(k+k_1)^{\frac{s}{r}}},\label{eq:err-2a}\\
E(U,\bar{u})&\le C_{\bar{u}}^{(1)}(k) \left(\Delta t + \Delta x\right) + \frac{C_{\bar{u}}^{(2)}}{(k+k_1)^{\frac{s}{r}}}. \label{eq:err-2b}
\end{align}
\end{subequations}
 \end{theorem}

One may expect a second-order accuracy in $\Delta x$ when no advection terms are present in \eqref{CH_MFG_eq}, e.g., when $h\equiv 0$. 
This, however, requires a stronger Lipschitz assumption on $f$. 
\begin{description} 
\item[(f-L$^\prime$)] There exist constants $C_0$, $L_f > 0$ and $\alpha_0 \in (0, 1)$ satisfying: 
\begin{subequations}
\begin{align}
|f(t_2, x, m) - f(t_1, x, m)| &\le C_0 |t_2-t_1|^{\alpha_0},\label{eq:f-L2a}\\
|f(t, x, m_2) - f(t, x, m_1)| &\le L_f |m_2(t,x) - m_1(t,x)|\label{eq:f-L2c}
\end{align}
for $(t_1, x),(t_2, x),(t,x)\in Q$, and $m,m_1,m_2\in \mathcal{D}_0(\T^d)$. 
Moreover, $f$ is Lipschitz continuous in $x$ in exactly the same sense as in (f-L).
\end{subequations}
\end{description}

\begin{theorem} 
\label{I23_thm3.5b}
Suppose that \textup{(f-B)}, \textup{(f-M)}, \textup{(f-P)}, and \textup{(TIV)} are satisfied. 
Assume that \textup{(H)} is satisfied for $h\equiv 0$ and that \textup{(f-L$^\prime$)} is satisfied instead of \textup{(f-L)}. Choose $\delta_k$, $r$ and $s$ by \eqref{eq:predefine}, \eqref{eq:val-p} and \eqref{eq:val-s}, respectively. 
Finally, assume that \eqref{eq:C24} holds. Then, 
we have
\begin{subequations}
\label{eq:err-3}
\begin{align}
I(\bar{M}^k, \bar{m}) &\le C_{\bar{m}}^{(1)}(k) \left[\Delta t + (\Delta x)^2\right] + \frac{C_{\bar{m}}^{(2)}}{(k+k_1)^{\frac{s}{r}}},\\
E(U,\bar{u})&\le C_{\bar{u}}^{(1)}(k) \left[\Delta t + (\Delta x)^2\right] + \frac{C_{\bar{u}}^{(2)}}{(k+k_1)^{\frac{s}{r}}}
\end{align}
\end{subequations}
instead of \eqref{eq:err-2}. 
\end{theorem}

The proofs of Theorems \ref{I23_thm3.5}, \ref{I23_thm3.5a}, and \ref{I23_thm3.5b} will be presented in Section~\ref{sec:proof}.

\begin{rem}
\label{rem:const-s}
The constant $\eta$ is chosen as the minimum of 
$\alpha_0$ in (TIV), $\beta$ in Proposition~\ref{prop:initial}, 
$\alpha$ in Proposition~\ref{B21_thm1}, and $\eta$ in Proposition~\ref{regularity_phi_psi}.
\end{rem}
\begin{rem}
  \label{rem:convergence}
\new{From the above main results, we obtain, in particular, the following convergence:
$$
\lim_{k \to \infty} \lim_{\Delta t, \Delta x \to +0} I(\bar{M}^k, \bar{m}) = 0, \qquad 
\lim_{k \to \infty} \lim_{\Delta t, \Delta x \to +0} E(U^k,\bar{u}) = 0.
$$
Although similar results have been obtained in \cite[Theorem 3.5]{I23}, we have succeeded in deriving more concrete error estimates. However, the convergence is not uniform with respect to the iteration because both $C_{\bar{m}}^{(1)}(k)$ and $C_{\bar{u}}^{(1)}(k)$ diverge as $k\to \infty$.} On the other hand, as explained in Remark \ref{rem:log}, under a suitable condition, $\bar{C}_{\bar{q}}^{(1)}(k)$ ($\mathrm{q} = \bar{m}, \bar{u}$) is uniformly bounded with respect to $k$. In this case, the error is theoretically guaranteed to converge as $k\to\infty$ and $\Delta t,\Delta x\to0$, independently of any relation between $k$ and the discretization parameters $\Delta t$ and $\Delta x$. However, the condition required to ensure the uniform boundedness of $\bar{C}_{\mathrm{q}}^{(1)}(k)$ is generally rather restrictive and may be difficult to verify in practice (see Remark \ref{rem:log} for details).
\end{rem}

\begin{rem}
\new{In connection with Remark \ref{rem:convergence}, we briefly discuss how the iteration number $k$ and the mesh sizes $\Delta t, \Delta x$ should be chosen in the case where $C_{\bar{m}}^{(1)}(k)$ grows polynomially in $k$.} Choose $a>0$ such that
$$
C_{\bar{m}}^{(1)}(k)\le Ck^a,
$$
and set $b:=s/r$. The estimate in Theorem \ref{I23_thm3.5} then becomes
$$
I(\bar{M}^k,\bar{m}) \le C k^a\left[(\Delta t)^{\eta/2}+(\Delta x)^\eta\right] + Ck^{-b}.
$$
Balancing the discretization and iteration terms suggests choosing the mesh sizes so that $(\Delta t)^{\eta/2}+(\Delta x)^\eta$ is of the same order as $k^{-(a+b)}$. In particular, when $\Delta t$ is chosen to be of order $(\Delta x)^2$, consistently with the CFL condition \eqref{eq:cfl}, $k$ should be chosen to be of order $(\Delta x)^{-\eta/(a+b)}$. Equivalently, for a target accuracy $\eps>0$, the estimate suggests choosing $k$ and $\Delta x$ to be of orders
$$
\eps^{-1/b} \quad \text{and} \quad \eps^{(a+b)/(b\eta)},
$$
respectively.

This scaling is only a sufficient condition derived from the worst-case estimate obtained in this paper. As shown in Proposition \ref{I23_thm3.4}, the exponent $a$ yielded by the present analysis is
$$
a = \max\{1, k_2(\bar{C}_2-1)\},
$$
where $k_2$ and $\bar{C}_2$ are specified there. As explained in Remark \ref{rem:log} and Appendix \ref{sec:app}, this exponent can be extremely large. Consequently, within the range that can be guaranteed by the present theoretical analysis, the required coupling between $k$ and the mesh sizes becomes extremely restrictive. In actual computations, however, such a stringent coupling does not appear to be necessary. Indeed, in the numerical experiments in Section \ref{sec:numerical_experiment}, no explicit coupling between $k$ and $\Delta x$ is imposed, yet the scheme remains stable.

The numerical results may also suggest that the non-uniform dependence on $k$ arising from the error analysis used in this paper, whose underlying mechanism is explained at the beginning of Section \ref{sec:proof}, is not necessarily intrinsic to the scheme. As mentioned in item 3 of Section \ref{sec:remarks}, although this direction is not pursued in the present paper, exploiting a variational structure for the discrete GCG scheme may lead to a more natural error estimate, possibly one that is uniform in $k$.
\end{rem}

\section{Preliminary results}
\label{sec:preliminary}

We review the convergence results on the GCG method established in \cite[Theorems 3.6 and 3.8]{NS26}.
In doing so, we first recall the variational formulations of the MFG system under the potential assumption \textup{(f-P)}. 
We introduce a functional: 
\[
\mathcal{J}(m, w)=\iint_Q m(t, x) L\left(t, x, \frac{w(t, x)}{m(t, x)}\right) ~dxdt 
+ \int_{\T^d} g(x) m(T, x) ~ dx+\int_0^T F(t, m(t,\cdot))~dt,
\]
where $L:Q\times \mathbb{R}^d\to\mathbb{R}$ is given as
\begin{equation*}
L(t, x, v) = \frac{1}{2} |v - h(t, x)|^2,
\end{equation*}
and the value of $m L(\cdot,\cdot,w/m) $ at points where $m = 0$ is understood as $0$.
The relationship between $H$ and $L$ is expressed as
\begin{equation}
\label{eq:legH}
H(t, x, p) = \sup_{v \in \R^d} \left[- p \cdot v - L(t, x, v)\right].
\end{equation}
As a matter of fact, $L$ is called the running cost and plays a role in measuring the cost of the control inputs in the FP equation. 
The functional $\mathcal{J}$ is defined in 
\[
\mathcal{R} := \left\{(m, w) \in W^{1, 2, q}(Q) \times \Theta \mid \begin{matrix}
\del_t m - \nu\Delta{m} + \nabla \cdot w = 0\mbox{ in }Q, \ m(0,\cdot) = m_0 \mbox{ on }\T^d,\\
\exists v \in L^{\infty}(Q; \R^d) \text{ s.t. } w = mv
\end{matrix}\right\}.
\]
Let $(\bar{m},\bar{u})$ be the solution of the MFG system \eqref{MFG_eq} and set $\bar{w}=-\bar{m}\nabla_pH(t,x,\nabla\bar{u})$. Then, as is verified in \cite[Lemma 3.1]{NS26} and \cite[Proposition 2]{LP23}, $(\bar{m},\bar{w})$ satisfies
\begin{equation*} 
\mathcal{J}(\bar{m}, \bar{w})=
\min_{(m, w) \in \mathcal{R}}\mathcal{J}(m, w). 
\end{equation*}
The GCG method \eqref{GCG} also has a variational formulation; see \cite[Proposition 3.2]{NS26} and \cite[Lemma 6]{LP23}. 
However, we do not recall it here, because the variational structure will not be used directly below. 


\begin{rem}
  \label{rem:m0-1}
We have $(\bar{m}_0,\bar{m}_0v)\in\mathcal{R}$, where $\bar{m}_0$ denotes the function described in Proposition~\ref{prop:initial}, and $v \in L^{\infty}(Q; \R^d)$ is a function satisfying the conditions specified by $\mathcal{R}$
\end{rem}

The optimality gap is defined by 
\begin{equation}
\eps_k := \mathcal{J}(\bar{m}_k, \bar{w}_k) - \mathcal{J}(\bar{m}, \bar{w}),
\label{eq:optimalitygap}
\end{equation}
which is obviously a positive number. 

In \cite{NS26}, we proved the following convergence results on the GCG method; see \cite[Theorems 3.6 and 3.8]{NS26}. In fact, \cite{NS26} extends the convergence analysis of the GCG method for global coupling established in \cite{LP23} to the case of local coupling. Although the overall analytical strategy is the same, the derivation of the individual estimates and the choice of appropriate norms require several original technical ideas. 

Moreover, we note again that the assumptions \textup{(f-P)} and \textup{(f-M)} are needed only for the PDE-level convergence results for the GCG method recalled below. Apart from the use of these results, the arguments in the present paper do not rely on \textup{(f-P)} or \textup{(f-M)}. In particular, the role of \textup{(f-P)} in the results below is essential, since the convergence estimates are obtained through the optimality gap $\eps_k$, which is defined in terms of the potential functional $\mathcal{J}$.

\begin{prop} 
\label{LP23_thm7} 
Suppose that 
\textup{(H)}, 
\textup{(f-B)}, 
\textup{(f-L)}, 
\textup{(f-M)}, 
\textup{(f-P)}, and \textup{(TIV)} are satisfied. 
Let $\bar{m}_{k}$, $m_k$, $u_k$ and $\gamma_k$ be the sequences generated by the GCG method \eqref{GCG} for a given $\bar{m}_{0}$ which is the function in Proposition~\ref{prop:initial}. 
Let $r \in \mathbb{R}$ be the constant appearing in Theorem \ref{I23_thm3.5b}. 
Then, there exists a constant $C > 0$ such that
\begin{equation}
\label{eq:thm7}
\|\bar{m}_{k}-\bar{m}\|_{{L^2(0, T; L^{\infty}(\T^d))}}
+\|u_k-\bar{u}\|_{L^{\infty}(Q)} 
+\|\gamma_k - \bar{\gamma}\|_{{L^2(0, T; L^{\infty}(\T^d))}} 
\le C \eps_k^{{\frac{1}{r}}}.  
\end{equation}
Moreover, we have $\eps_k\le C$. 
\end{prop}

\begin{prop} \label{LP23_thm8-2}
Suppose that 
\textup{(H)}, 
\textup{(f-B)}, 
\textup{(f-L)}, 
\textup{(f-M)}, 
\textup{(f-P)}, and \textup{(TIV)} are satisfied. 
Suppose that $\delta_k$ is chosen in accordance with \eqref{eq:predefine}. 
Then, there exists a constant $C_{\ast} > 0$ satisfying
\begin{equation} 
\label{thm8-2_eq}
\eps_{k+1} \le (1-\delta_k) \eps_k + C_{\ast} \delta_k^2 \eps_k^{\frac{2}{r(r-1)}}.
\end{equation}
Moreover, we have
\begin{equation} 
\label{thm8-2_eq_2}
\eps_k \le \frac{N}{(k + k_1)^s},
\end{equation}
where $s$ is chosen by \eqref{eq:val-s}, and $N$ is a constant defined by $N = \eps_0 k_1^{k_2} \exp{\left(C_{\ast} k_2^2(k_1+1)/k_1^2\right)}$ if $d = 1$, and by
$$
N = \max{\left\{\eps_0 k_1^s, \left(\frac{C_{\ast} k_2^2}{k_2-s}\right)^{\frac{1}{r}}\right\}}, 
$$
if $d \ge 2$.
\end{prop}

In the remainder of this section, we establish the regularity and boundedness properties of the solutions $(\phi_k,\psi_k)$ of \eqref{CH_MFG_eq}.

\begin{prop}
\label{prop:phi_max_min}
Suppose that 
\textup{(H)}, 
\textup{(f-B)}, 
\textup{(f-L)}, 
\textup{(f-M)}, and \textup{(TIV)} are satisfied. Then, there exist positive constants $\phi_{\max}$ and $\phi_{\min}$ such that 
$$
\phi_{\min} \le \phi_k \le \phi_{\max} \mbox{ in } Q
$$
for $k=0,1,2,\ldots$. 
\end{prop}

\begin{proof}
As recalled in Proposition \ref{LP23_prop23}, we know $\|u_k\|_{L^\infty(Q)}\le C$. Hence, $\exp (-C/2\nu)\le \phi_k=\exp(-u_k/2\nu)\le \exp(C/2\nu)$. 
\end{proof}

\begin{prop} \label{regularity_phi_psi}
Suppose that 
\textup{(H)}, 
\textup{(f-B)}, 
\textup{(f-L)}, 
\textup{(f-M)}, and \textup{(TIV)} are satisfied. Then, there exist constants $\eta \in (0, 1), B_{\phi}$ and $B_{\psi} > 0$, which are independent of the iteration number, such that
\begin{equation} \label{regularity_phi_psi_eq}
\phi_k, \psi_k \in \mathcal{C}^{1 + \frac{\eta}{2}, 2 + \eta}(Q), \quad \|\phi_k\|_{\mathcal{C}^{1 + \frac{\eta}{2}, 2 + \eta}(Q)} \le B_{\phi}, \quad \|\psi_k\|_{\mathcal{C}^{1 + \frac{\eta}{2}, 2 + \eta}(Q)} \le B_{\psi} 
\end{equation}
for all $k = 0, 1, 2, \dots$.
\end{prop}

To prove this proposition, we use the following results. 

\begin{lemma} \label{B21_thm7}
Let $R>0$ and $\beta_0\in (0,1)$. Suppose that 
$u_0\in \mathcal{C}^{2+\beta_0}(\T^d)$, 
$b\in \mathcal{C}^{\frac{\beta_0}{2}, \beta_0}(Q; \R^d)$, $c\in \mathcal{C}^{\frac{\beta_0}{2}, \beta_0}(Q)$ and $w\in \mathcal{C}^{\frac{\beta_0}{2}, \beta_0}(Q)$ satisfies
$$
\|u_0\|_{\mathcal{C}^{2+\beta_0}(\T^d)} \le R,\quad \|b\|_{\mathcal{C}^{\frac{\beta_0}{2}, \beta_0}(Q; \R^d)}\le R,\quad \|c\|_{\mathcal{C}^{\frac{\beta_0}{2}, \beta_0}(Q)} \le R, \quad \|w\|_{\mathcal{C}^{\frac{\beta_0}{2}, \beta_0}(Q)} \le R.
$$ 
Then, the initial value problem 
\begin{equation} 
\label{parabolic_eq}
\del_t u - \nu \Delta{u} + b \cdot \nabla{u} + c u = w \quad \mbox{in } Q, \qquad u(0,\cdot) = u_0 \quad \mbox{on } \T^d
\end{equation}
has a unique solution $u \in \mathcal{C}^{1 + \beta/2, 2 + \beta}(Q)$ for some $\beta = \beta(\beta_0) \in (0, 1)$. Moreover, there exists a constant $C=C(\beta_0, R) > 0$ such that $\|u\|_{\mathcal{C}^{1 + \beta/2, 2 + \beta}(Q)} \le C$.
\end{lemma}

\begin{proof}
See {\cite[Theorem 7]{B21}}.
\end{proof}

\begin{lemma} 
\label{B21_lem12}
There exist $\delta \in (0, 1)$ and $C > 0$ such that for any $u \in W^{1, 2, q}(Q)$,
$$
\|u\|_{\mathcal{C}^{\delta}(Q)} + \|\nabla{u}\|_{\mathcal{C}^{\delta}(Q; \R^d)} \le C \|u\|_{W^{1, 2, q}(Q)}.
$$
\end{lemma}

\begin{proof}
This is a standard result. 
See \cite[pages 80 and 342]{lsu68} or \cite[Lemma 12]{B21} for example.
\end{proof}

\begin{proof}[Proof of Proposition \ref{regularity_phi_psi}]
By Proposition \ref{LP23_prop23}, we know $\|\bar{m}_k\|_{W^{1, 2, q}(Q)} \le C$. Hence, $\|\bar{m}_k\|_{\mathcal{C}^{\delta}(Q)}\le C$ for some $\delta \in (0, 1)$ and $C>0$ which are independent of $k$ in view of Lemma \ref{B21_lem12}. In particular,the H{\" o}lder constants of $\bar{m}_k$ (with respect to both variables $t$ and $x$) are bounded by some $C>0$ which is independent of $k$.

By choosing $b = -h, c = \gamma_k/2 \nu, w \equiv 0$, and $u_T = \exp{(-g/2\nu)}$ in Lemma \ref{B21_thm7}, $\phi_k(T-t,x)$ is the solution of \eqref{parabolic_eq}. From the Lipschitz continuity of $f$ and the H\"{o}lder continuity of $\bar{m}_k$, we have
\begin{align*}
& \left|f(t_2, x_2, \bar{m}_k(t_2,\cdot)) - f(t_1, x_1, \bar{m}_k(t_1,\cdot))\right| \\
&\le C\left(|t_2-t_1|^{\alpha_0} + |x_2-x_1|\right) + L_f \|\bar{m}_k(t_2,\cdot) - \bar{m}_k(t_1,\cdot)\|_{L^{\infty}(\T^d)}\\
&\le C \left(|t_2-t_1|^{\alpha_0} + |x_2-x_1|\right) + L_f C |t_2 - t_1|^{\delta} \\
&\le C(1 + L_f) \left(|t_2-t_1|^{\min{\{\alpha_0, \delta\}}} + |x_2-x_1|\right).
\end{align*}
Hence $c \in \mathcal{C}^{\min{\{\alpha_0, \delta\}}, 1}(Q)$. Since $g$ and $h$ are sufficiently smooth, choosing $\beta_1 \in (0, 1)$ appropriately, we have $h \in \mathcal{C}^{\beta_1/2, \beta_1}(Q; \R^d), c \in \mathcal{C}^{\beta_1/2, \beta_1}(Q)$, and $u_0 \in \mathcal{C}^{2+\beta_1}(\T^d)$. Therefore, there exist constants $\eta_1 = \eta_1(\beta_1) \in (0, 1)$ and $B_{\phi} > 0$ such that
$$
\phi_k \in \mathcal{C}^{1+\frac{\eta_1}{2}, 2 + \eta_1}(Q), \qquad \|\phi_k\|_{\mathcal{C}^{1 + \frac{\eta_1}{2}, 2 + \eta_1}(Q)} \le B_{\phi}.
$$
Particularly, since $\|c\|_{\mathcal{C}^{\beta_1/2, \beta_1}(Q)}$ and $u_0 = \exp{(-g/2 \nu)}$ are bounded by a constant independent of $k$, $B_{\phi}$ does not depend on $k$.

Next, by choosing $b = h, c = \nabla \cdot h + \gamma_k/2 \nu, w \equiv 0$, and $u_0 = m_0/\phi_k(0)$, $\psi_k$ is the solution to (\ref{parabolic_eq}). From $m_0 \in \mathcal{C}^{2+\alpha_0}(\T^d)$, $\phi_k(0) \in \mathcal{C}^{2+\eta_1}(\T^d)$, and $\phi_k(0) \ge \phi_{\min} > 0$, we have $u_0 \in \mathcal{C}^{2+\beta_2}(\T^d)$ for a sufficiently small $\beta_2 \in (0, \min\{\alpha_0, \beta_1, \eta_1\})$. We have also $b \in \mathcal{C}^{\beta_2/2, \beta_2}(Q; \R^d)$ and $c \in \mathcal{C}^{\beta_2/2, \beta_2}(Q)$; therefore, there exist constants $\eta_2 = \eta_2(\beta_2) \in (0, 1)$ and $B_{\psi} > 0$ such that
$$
\psi_k \in \mathcal{C}^{1+\frac{\eta_2}{2}, 2 + \eta_2}(Q), \qquad \|\psi_k\|_{\mathcal{C}^{1 + \frac{\eta_2}{2}, 2 + \eta_2}(Q)} \le B_{\psi}.
$$
Particularly, since $\|u_0\|_{\mathcal{C}^{2+\beta_2}(\T^d)}$ is bounded a constant depending only on the Lipschitz constants of $m_0 \in \mathcal{C}^{2+\alpha_0}(\T^d)$, $B_{\phi}$ and $\phi_{\min}$, $B_{\psi}$ does not depend on $k$.
Thus, taking $\eta = \min\{\eta_1, \eta_2\}$ concludes the proof of \eqref{regularity_phi_psi_eq}.
\end{proof}

\begin{rem}
In \cite{I23}, the conditions $\phi_k \in \mathcal{C}^{1+\eta/2, 2+\eta}(Q)$ and $\psi_k \in \mathcal{C}^{1+\eta/2, 2+\eta}(Q)$ are only assumed. In this respect, the results obtained in this paper are stronger than those of \cite{I23}.
\end{rem}

\section{Proof of Theorem \ref{I23_thm3.5}}
\label{sec:proof}

 Throughout this section, we suppose that 
\textup{(H)}, 
\textup{(f-B)}, 
\textup{(f-L)}, 
\textup{(f-M)}, \textup{(f-P)}, and \textup{(TIV)} are satisfied. 
Let $(\bar{m}, \bar{u})$ and 
$(u_k,m_k,\gamma_k,\bar{m}_{k})$ 
be the solutions of the MFG system \eqref{MFG_eq} and GCG method \eqref{GCG} given in Propositions in \ref{B21_thm1} and \ref{LP23_prop23}, respectively. 
The solution of the discrete GCG scheme \eqref{eq:dGCG} as stated in Theorem \ref{thm:DMP} is denoted by $(\bar{M}_{n,i}^k,{M}_{n,i}^k, \Phi_{n,i}^k,\Psi_{n,i}^k)$. 
Moreover, the solution of the CH-GCG method \eqref{CH_MFG_eq} is denoted by $(\phi_k, \psi_k)$.

Before deriving the error estimates, we first outline the strategy of the proof of Theorem \ref{I23_thm3.5}, in order to clarify why the main estimate contains an error term of the form
$$
C^{(1)}_{\mathrm{q}}(k) \left[(\Delta t)^{\eta/2}+(\Delta x)^\eta\right] \qquad (\mathrm{q} = \bar{m}, \bar{u}).
$$
The factor $C^{(1)}_{\mathrm{q}}(k)$ appears because the discretization errors are propagated through the GCG iterations.

To see this structure, we write the density error as
$$
\bar{M}^k-\bar{m}=(\bar{M}^k-\bar{m}_k)+(\bar{m}_k-\bar{m}).
$$
That is, the total error is decomposed into the discretization error $\bar{M}^k-\bar{m}_k$ and the iteration error $\bar{m}_k-\bar{m}$. The latter is controlled by the convergence results for the PDE-level GCG method, namely Propositions \ref{LP23_thm7} and \ref{LP23_thm8-2}. Therefore, the main task in this section is to estimate the discretization error $\bar{M}^k-\bar{m}_k$. This estimate is nontrivial because the discrete scheme \eqref{eq:dGCG} is not obtained simply by discretizing a PDE with a fixed source term. Indeed, in the PDE-level iteration \eqref{CH_MFG_eq}, the coupling term is given by $\gamma_k=f(\cdot,\cdot,\bar{m}_k)$. In the numerical scheme \eqref{eq:dGCG}, however, $\bar{m}_k$ is not available, and the coupling term is replaced by
$$
\Gamma^k_{n,i}=f(t_n,x_i,\bar{M}^k_n).
$$
Thus, the discrepancy between $\Gamma^k$ and $\gamma_k$ contains not only the time and space discretization errors but also the density-iterate error $\bar{M}^k-\bar{m}_k$.

This dependence is reflected in the estimates for the Cole--Hopf variables. Proposition \ref{I23_prop4.4,4.6} shows that, schematically,
$$
E(\Phi^k,\phi_k)+E(\Psi^k,\psi_k)
\le C\left[(\Delta t)^{\eta/2}+(\Delta x)^\eta+E(\bar{M}^k,\bar{m}_k)\right].
$$
Therefore, the error in the density iterate is transmitted to the Cole--Hopf variables through the coupling term.

It remains to estimate $E(\bar{M}^k,\bar{m}_k)$ itself. This is done in Proposition \ref{I23_thm3.4}, where the above estimates for the Cole--Hopf variables are combined with the density update. The resulting estimate accounts for the propagation of errors over the GCG iterations and yields
$$
E(\bar{M}^k,\bar{m}_k)\le K(k)\left[(\Delta t)^{\eta/2}+(\Delta x)^\eta\right].
$$
Consequently, the same iteration-dependent discretization error appears in the estimates for $E(\Phi^k,\phi_k)$ and $E(\Psi^k,\psi_k)$ as well. This recursive dependence is the mechanism behind the factor $C^{(1)}_{\mathrm{q}}(k)$ in Theorem \ref{I23_thm3.5}.

We now begin with the estimate for the Cole--Hopf variables $\phi_k$ and $\psi_k$.

\begin{prop} \label{I23_prop4.4,4.6}
There exist positive constants $C_1, C_2, C_3$, and $C_4$\new{, independent of both the discretization parameters $\Delta t, \Delta x$ and the iteration number $k$,} such that
\begin{subequations}
\label{eq:e51}  
\begin{align}
E(\Phi^k,\phi_k) &\le C_1 \left[(\Delta t)^{\frac{\eta}{2}} + (\Delta x)^{\eta}\right] + C_2 E(\bar{M}^k, \bar{m}_k),\label{eq:e51a} \\
E(\Psi^k,\psi_k) &\le C_3 \left[(\Delta t)^{\frac{\eta}{2}} + (\Delta x)^{\eta}\right] + C_4 E(\bar{M}^k, \bar{m}_k).\label{eq:e51b} 
\end{align}
\end{subequations}
Moreover, if $\phi_k$ and $\psi_k$ have the regularity property \eqref{eq:C24}, then we have
\begin{subequations}
\label{eq:e52}  
\begin{align}
E(\Phi^k,\phi_k)&\le C_1 (\Delta t + \Delta x) + C_2 E(\bar{M}^k, \bar{m}_k),\label{eq:e52a}\\
E(\Psi^k,\psi_k) &\le C_3 (\Delta t + \Delta x) + C_4 E(\bar{M}^k, \bar{m}_k).\label{eq:e52b}
\end{align}
\end{subequations}
Finally, in addition to the above, assume that $h\equiv 0$ in \textup{(H)} and that \textup{(f-L$^\prime$)} is satisfied, then 
\begin{subequations}
\label{eq:e53}  
\begin{align}
E(\Phi^k,\phi_k)&\le C_1 \left[\Delta t + (\Delta x)^2\right] + C_2 E(\bar{M}^k, \bar{m}_k),\label{eq:e53a} \\
E(\Psi^k,\psi_k) &\le C_3 \left[\Delta t + (\Delta x)^2\right] + C_4 E(\bar{M}^k, \bar{m}_k).\label{eq:e53b} 
\end{align}
\end{subequations}
\end{prop}

\begin{proof}
Although the proof of estimates \eqref{eq:e51} is essentially reported in \cite[Propositions 4.4 and 4.6]{I23} only for $d=1$, we state it here for completeness.

Setting $\phi^k_{n, i} := \phi_k(t_n, x_i)$ and $e^k_{n, i} := \Phi^k_{n, i} - \phi^k_{n, i}$, we see by \eqref{GCGa} and \eqref{eq:dGCG2}
\begin{align} \label{eq_consistency_error}
& \frac{e^k_{n, i} - e^k_{n-1, i}}{\Delta t} + \nu D_2(e^k, n, i) + D_1^1(e^k, n, i) - \frac{1}{2 \nu} \Gamma^k_{n, i} e^k_{n-1, i}\notag\\
&= \underbrace{\left[\del_t \phi_k(t_n, x_i) + \nu \Delta \phi_k(t_n, x_i) + h_{n, i} \cdot \nabla \phi_k(t_n, x_i) - \frac{1}{2\nu} \gamma_k(t_n, x_i) \phi_k(t_n, x_i)\right]}_{=0}\notag\\
&\qquad +\underbrace{\frac{\Phi^k_{n, i} - \Phi^k_{n-1, i}}{\Delta t} + \nu D_2(\Phi^k, n, i) + D_1^1(\Phi^k, n, i) - \frac{1}{2 \nu} \Gamma^k_{n, i} \Phi^k_{n-1, i}}_{=0}\notag\\
&\qquad -\left[\frac{\phi^k_{n, i} - \phi^k_{n-1, i}}{\Delta t} + \nu D_2(\phi^k, n, i) + D_1^1(\phi^k, n, i) - \frac{1}{2 \nu} \Gamma^k_{n, i} \phi^k_{n-1, i}\right]\notag\\
&= \left[\del_t \phi_k(t_n, x_i) - \frac{\phi^k_{n, i} - \phi^k_{n-1, i}}{\Delta t}\right] + \nu \left[\Delta \phi_k(t_n, x_i) - D_2(\phi^k, n, i)\right]\notag\\
&\qqquad +\left[h_{n, i} \cdot \nabla \phi_k(t_n, x_i) - D_1^1(\phi^k, n, i)\right] \notag\\
&\qqquad -\frac{1}{2 \nu} \left[\gamma_k(t_n, x_i)(\phi^k_{n, i} - \phi^k_{n-1, i}) + (\gamma_k(t_n, x_i) - \Gamma^k_{n, i})\phi^k_{n-1, i} \right]\notag\\
&=: r^1_{n, i} + \nu r^2_{n, i} + r^3_{n, i} -\frac{1}{2 \nu} r^4_{n, i},
\end{align}
where we have used the fact that $\phi_k$ and $\Phi^k$ solve \eqref{GCGa} and \eqref{eq:dGCG2}, respectively. By a direct use of Taylor's theorem, we derive
\begin{subequations}
\label{eq:quotients}
\begin{gather}
|r^1_{n, i}| \le C B_\phi (\Delta t)^{\frac{\eta}{2}}, \qquad |r^2_{n, i}| \le CB_\phi (\Delta x)^\eta,\label{eq:quotients1} \\
|r^3_{n, i}| \le C\|h\|_{L^{\infty}(Q;\mathbb{R}^d)} B_\phi \Delta x,\label{eq:quotients2}
 \end{gather}
\end{subequations}
where $B_\phi$ denotes the constant appearing in Proposition \ref{regularity_phi_psi}. 

Moreover, using \textup{(f-L)}, \eqref{eq:f-L1a}, we have 
\begin{align} \label{ineq_estimate_r4_1}
|r^4_{n, i}| &\le \|\gamma_k\|_{L^{\infty}(Q)} \|\del_t \phi_k\|_{L^{\infty}(Q)} \Delta t + \left|f(t_n, x_i, \bar{m}_k(t_n)) - f(t_n, x_i, \bar{M}^k_n)\right| \|\phi_k\|_{L^{\infty}(Q)} \notag\\
&\le \|\gamma_k\|_{L^{\infty}(Q)} \|\del_t \phi_k\|_{L^{\infty}(Q)} \Delta t + L_f \|\phi_k\|_{L^{\infty}(Q)} \left\|\bar{m}_k(t_n) - \bar{M}^k_n\right\|_{L^{\infty}(\T^d)}.
\end{align}

We estimate the second term of the right-hand side as 
\begin{align*}
|\bar{m}_k(t_n, x) - \bar{M}^k_n(x)| &\le |\bar{m}_k(t_n, x) - \bar{m}_k(t_n, x_i)| + |\bar{m}_k(t_n, x_i) - \bar{M}^k_{n, i}| \\
&\le C\|\nabla \bar{m}_k\|_{L^\infty(Q)} \Delta x + E(\bar{M}^k, \bar{m}_k)
\end{align*}
for $x \in \prod_{\ell=1}^d \left[x_{i_{\ell}} - \Delta x/2, x_{i_{\ell}} + \Delta x/2\right)$. 

Since 
$\|\gamma_k\|_{L^{\infty}(Q)}\le C$ and 
$\|\nabla \bar{m}_k\|_{L^\infty(Q)}\le \|\bar{m}_k\|_{W^{1,2,q}(Q)}\le C$ by Proposition \ref{LP23_prop23} and Lemma \ref{B21_lem12}, we obtain
\begin{equation*}
|r^4_{n, i}| \le C B_\phi \Delta t + C B_\phi \left[\Delta x + E(\bar{M}^k, \bar{m}_k)\right].
\end{equation*}
\new{We note here that the a priori bounds $B_{\phi}$ and $B_{\psi}$ for $\phi$ and $\psi$ (Proposition \ref{regularity_phi_psi}), as well as the uniform bounds on $\gamma_k$ and $\bar{m}_k$ (Proposition \ref{LP23_prop23}), depend only on the continuous problem data. Therefore, combining the estimates for $r^1_{n, i}, \dots, r^4_{n, i}$, we see that there exist positive constants $C$ and $C'$, independent of the discretization parameters $\Delta t, \Delta x$ and the iteration number $k$, such that}
\begin{equation}
  \label{ineq_estimate_r4_1b}
\left|r^1_{n, i} + \nu r^2_{n, i} + r^3_{n, i} -\frac{1}{2 \nu} r^4_{n, i}\right| \le C \left[(\Delta t)^{\frac{\eta}{2}}+(\Delta x)^{\eta}\right] + C' E(\bar{M}^k, \bar{m}_k).
\end{equation}

Letting $\delta_1 := \Delta t/\Delta x$ and $\delta_2 := \Delta t/(\Delta x)^2$, we now rewrite \eqref{eq_consistency_error} as
\begin{align*}
e^k_{n-1, i} &= \frac{1}{1 + \frac{\Delta t}{2 \nu} \Gamma^k_{n, i}} \left[\left(1- \sum_{\ell=1}^d |h_{n, i}^{\ell}| \delta_1 - 2 d \nu \delta_2\right)e^k_{n, i} + \sum_{\ell=1}^d \left(h_{n, i}^{\ell, +} \delta_1 + \nu \delta_2\right)e^k_{n, i[\ell^+]} \right.\\
&\qqquad + \left. \sum_{\ell=1}^d \left(h_{n, i}^{\ell, -} \delta_1 + \nu \delta_2\right)e^k_{n, i[\ell^-]} -\left(r^1_{n, i} + \nu r^2_{n, i} + r^3_{n, i} -\frac{1}{2 \nu} r^4_{n, i}\right) \Delta t\right],
\end{align*}
where the denominator $\ge 1$ thanks to \textup{(f-B)}. 

Setting $E^k_n := \max_{i\in\Lambda_x} |e^k_{n, i}|$, we can perform the estimation as
\begin{align*}
|e^k_{n-1, i}| &\le \frac{1}{1 + \frac{\Delta t}{2 \nu} \Gamma^k_{n, i}} \left[\left(1- \sum_{\ell=1}^d |h_{n, i}^{\ell}| \delta_1 - 2 d \nu \delta_2\right) E^k_n + \sum_{\ell=1}^d \left(h_{n, i}^{\ell, +} \delta_1 + \nu \delta_2\right) E^k_n \right.\\
&\ + \left. \sum_{\ell=1}^d \left(h_{n, i}^{\ell, -} \delta_1 + \nu \delta_2\right) E^k_n +\left\{
C\left[(\Delta t)^{\frac{\eta}{2}}+(\Delta x)^{\eta}\right] + C' E(\bar{M}^k, \bar{m}_k)\right\}\Delta t\right]\\
&\le E^k_n + \left\{
C\left[(\Delta t)^{\frac{\eta}{2}}+(\Delta x)^{\eta}\right] + C' E(\bar{M}^k, \bar{m}_k)\right\}\Delta t.
\end{align*}
Therein, we have used inequalities $1 - |h_{n, i}^{\ell}|\delta_1 - 2 d \nu \delta_2 > 0$ and $1 + \frac{\Delta t}{2 \nu} \Gamma^k_{n, i} > 1$, which are the consequences \eqref{eq:cfl} and \textup{(f-B)}, respectively. Therefore, 
\begin{align*}
E^k_{n-1} &\le E^k_n + \left\{
C\left[(\Delta t)^{\frac{\eta}{2}}+(\Delta x)^{\eta}\right] + C' E(\bar{M}^k, \bar{m}_k)\right\} \Delta t\\
&\le E_{N_t} + (N_t-n+1) \Delta t \left\{
C\left[(\Delta t)^{\frac{\eta}{2}}+(\Delta x)^{\eta}\right] + C' E(\bar{M}^k, \bar{m}_k)\right\}\\
&\le T \left\{
C\left[(\Delta t)^{\frac{\eta}{2}}+(\Delta x)^{\eta}\right] + C' E(\bar{M}^k, \bar{m}_k)\right\},
\end{align*}
which implies \eqref{eq:e51a}. The estimate \eqref{eq:e51b} is derived in exactly same way. 

If $\phi_k$ and $\psi_k$ have the regularity property \eqref{eq:C24}, we derive
\begin{equation}
\label{eq:quotients3}
|r^1_{n, i}| \le C \Delta t, \qquad |r^2_{n, i}| \le C(\Delta x)^2
\end{equation}
instead of \eqref{eq:quotients1}. The rest is exactly the same as above. Thus, \eqref{eq:e52} are proved.

Finally, assume that $h\equiv 0$ in \textup{(H)} and suppose that \textup{(f-L$^\prime$)} is satisfied. 
Then, \eqref{eq:quotients3} is still available and $r^3_{n, i} = 0$. Moreover, using \textup{(f-L$^\prime$)}, we can directly estimate as
\begin{align} \label{ineq_estimate_r4_2}
|r^4_{n, i}| &\le \|\gamma_k\|_{L^{\infty}(Q)} \|\del_t \phi_k\|_{L^{\infty}(Q)} \Delta t + L_f \|\phi_k\|_{L^{\infty}(Q)} \left|\bar{m}_k(t_n, x_i) - \bar{M}^k_{n, i}\right|\\
&\le \|\gamma_k\|_{L^{\infty}(Q)} \|\del_t \phi_k\|_{L^{\infty}(Q)} \Delta t + L_f \|\phi_k\|_{L^{\infty}(Q)} E(\bar{M}^k, \bar{m}_k).\notag
\end{align}
instead of \eqref{ineq_estimate_r4_1}. Therefore, \eqref{ineq_estimate_r4_1b} is replaced by 
\begin{equation*}
\left|r^1_{n, i} + \nu r^2_{n, i} -\frac{1}{2 \nu} r^4_{n, i}\right| \le C \left[\Delta t+(\Delta x)^2\right] + C' E(\bar{M}^k, \bar{m}_k).
\end{equation*}
Consequently, we prove \eqref{eq:e53}. 
\end{proof}

We shall use the following elementary lemma in the proof:

\begin{lemma}
\label{lem:recurrence_ineq}
Let $a_k$ be a sequence satisfying the following recurrence formula
$$
a_0 = 0, \qquad a_{k+1} = \left[(1-\delta_k) + c_0 \delta_k \right] a_k + \delta_k
$$
for some constant $c_0 > 1$. Then, we have
\begin{equation}
\label{ineq:recurrence_ineq}
a_k \le k_2 k^{\max\{1, k_2(c_0-1)\}}.
\end{equation}
Moreover, $a_k$ is logarithmically bounded from below.
\end{lemma}

\begin{proof}
The logarithmic lower bound on $a_k$ follows immediately from $a_{k+1} \ge a_k + \delta_k$, which in turn implies
$$
a_k \ge \sum_{\ell=0}^{k-1} \frac{k_2}{\ell+k_1}.
$$

Next, we prove by induction that $a_k \le c k^s$ for suitable constants $c$ and $s$. We will see that one can take $c=k_2$ and $s=\max\{1, k_2(c_0-1)\}$, as stated.

It obviously holds for the case $k=0$. In general, assuming that it holds for $k$, we see
\begin{align*}
a_{k+1} &\le \left[(1-\delta_k) + c_0 \delta_k\right] c k^s + \delta_k\\
&= \new{c(k+1)^s} + \frac{c}{k+k_1} \left\{\left[k + \new{k_1} + k_2(c_0-1)\right] k^s - (k+1)^s (k+k_1) + \frac{k_2}{c}\right\}\\
&=: \new{c(k+1)^s} + \frac{c}{k+k_1} P(k).
\end{align*}
Since $k_1 \ge 1$, 
\begin{align*}
P(k) &= - \left\{\left[(k+1)^s - k^s\right] (k+k_1) - k_2 (c_0-1) k^s - \frac{k_2}{c}\right\}\\
&\le - \left\{(k+1)^{s+1} - k^{s+1} - \left[k_2 (c_0-1) + 1\right] k^s - \frac{k_2}{c}\right\}\\
&= -\left[Q(k) - \frac{k_2}{c}\right],
\end{align*}
where $
Q(k) := (k+1)^{s+1} - k^{s+1} - \left[k_2 (c_0-1) + 1\right] k^s$. 
We observe
\[
Q'(k) 
= (s+1) k^s \left\{\left(1+\frac{1}{k}\right)^s - 1 - \frac{k_2(c_0-1) + 1}{s+1} \frac{s}{k} \right\}.
\]
At this stage, take $s = \max\{1, k_2 (c_0 - 1)\}$. Then, since
$$
\left(1+\frac{1}{k}\right)^s-1 \ge \frac{s}{k} \quad \text{and}\quad \frac{s}{k}\geq\frac{k_2(c_0-1)+1}{s+1}\frac{s}{k}
$$
(the first inequality follows by comparing the convex function $x\mapsto(1+x)^s$ with its tangent line at $x=0$), we obtain $Q'(k)>0$, which implies that $Q$ is increasing on $[1, \infty)$. We then obtain
$$
Q(k) \ge \min{\{Q(0), Q(1)\}} = \min{\{1, 2^{s+1} - s - 2\}} = 1,
$$
which implies that $P(k) \le k_2/c - 1$. Consequently, by choosing $c = k_2$, we have $P(k) \le 0$ and hence $\new{a_{k+1} \le c (k+1)^s}$. Therefore, we obtain \eqref{ineq:recurrence_ineq}.
\end{proof}

\begin{prop} \label{I23_thm3.4}
There exist increasing functions $K_{\mathrm{q}}(k) > 0$ of $k$ ($\mathrm{q} = \phi, \psi, \bar{m}$) such that
\begin{align*}
E(\Phi^k,\phi_k) &\le K_{\phi}(k) \left[(\Delta t)^{\frac{\eta}{2}} + (\Delta x)^{\eta}\right],\\
E(\Psi^k,\psi_k) &\le K_{\psi}(k) \left[(\Delta t)^{\frac{\eta}{2}} + (\Delta x)^{\eta}\right],\\
E(\bar{M}^k,\bar{m}_k) &\le K_{\bar{m}}(k) \left[(\Delta t)^{\frac{\eta}{2}} + (\Delta x)^{\eta}\right].
\end{align*}
The asymptotic growth of $K_{\mathrm{q}}(k)$ ($\mathrm{q} = \phi, \psi, \bar{m}$) is polynomially bounded from above as $k\to\infty$. Moreover, if $\phi_k$ and 
$\psi_k$ have the regularity property \eqref{eq:C24}, 
$(\Delta t)^{\frac{\eta}{2}}+(\Delta x)^{\eta}$ can be replace by 
$\Delta t+\Delta x$ in these inequalities. 
Finally, in addition to the above assumptions, suppose that $h\equiv 0$ in \textup{(H)} and that \textup{(f-L$^\prime$)} is satisfied. Then, $(\Delta t)^{\frac{\eta}{2}}+(\Delta x)^{\eta}$ can be replaced by $\Delta t+(\Delta x)^2$. 
\end{prop}

\begin{proof}
We set 
$E^k := E(M^k,m_k)$ and $\bar{E}^k :=E(\bar{M}^k,\bar{m}_k)$.
Using Proposition \ref{I23_prop4.4,4.6}, \eqref{eq:e51} , we estimate as
\begin{align*}
E^k &= \max_{(n, i)\in\Lambda} |\Phi^k_{n, i} \Psi^k_{n, i} - \phi_k(t_n, x_i) \psi_k(t_n, x_i)|\\
&\le \max_{(n, i)\in\Lambda} \left[|\Phi^k_{n, i} - \phi_k(t_n, x_i)| \Psi^k_{n, i} + \phi_k(t_n, x_i) |\Psi^k_{n, i} - \psi_k(t_n, x_i)|\right]\\
&\le \Psi_{\max} \left\{C_1 \left[(\Delta t)^{\frac{\eta}{2}} + (\Delta x)^{\eta}\right] + C_2 E(\bar{M}^k, \bar{m}_k)\right\} \\
& \qqquad + \phi_{\max} \left\{C_3 \left[(\Delta t)^{\frac{\eta}{2}} + (\Delta x)^{\eta}\right] + C_4 E(\bar{M}^k, \bar{m}_k)\right\}\\
&= \bar{C}_1 \left[(\Delta t)^{\frac{\eta}{2}} + (\Delta x)^{\eta}\right] + \bar{C}_2 \bar{E}^k,
\end{align*}
where $\bar{C}_1 := C_1 \Psi_{\max} + C_3 \Phi_{\max}$ and $\bar{C}_2 := C_2 \Psi_{\max} + C_4 \Phi_{\max}$. \new{(Since $\Phi_{\max}$ and $\Psi_{\max}$ are independent of $\Delta t, \Delta x$, and $k$, so are the new constants $\bar{C}_1$ and $\bar{C}_2$.)} Repeating this procedure, we observe
\begin{align*}
\bar{E}^{k+1} &\le (1-\delta_k) \bar{E}^k + \delta_k E^k\\
&\le \left[(1-\delta_k) + \delta_k \bar{C}_2\right] \bar{E}^k + \delta_k \bar{C}_1 \left[(\Delta t)^{\frac{\eta}{2}} + (\Delta x)^{\eta}\right]\\
&\le \bar{C}_1 \left[(\Delta t)^{\frac{\eta}{2}} + (\Delta x)^{\eta}\right] \sum_{\ell=0}^k \delta_{\ell} \prod_{\ell'=\ell+1}^k \left[(1-\delta_{\ell'}) + \delta_{\ell'} \bar{C}_2\right]\\
&= K_{\bar{m}}(k+1) \left[(\Delta t)^{\frac{\eta}{2}} + (\Delta x)^{\eta}\right],
\end{align*}
where we have used $\bar{E}_0=0$ and have set 
$$
K_{\bar{m}}(k) := \bar{C}_1 \sum_{\ell=0}^{k-1} \delta_{\ell} \prod_{\ell'=\ell+1}^{k-1} \left[(1-\delta_{\ell'}) + \delta_{\ell'} \bar{C}_2\right].
$$
This, together with Proposition \ref{I23_prop4.4,4.6}, we deduce
\begin{align*}
E(\Phi^k,\phi_k) &\le \underbrace{\left[C_1 + C_2 K_{\bar{m}}(k)\right]}_{=: K_{\phi}(k)} \left[(\Delta t)^{\frac{\eta}{2}} + (\Delta x)^{\eta}\right],\\
E(\Psi^k,\psi_k) &\le \underbrace{\left[C_3 + C_4 K_{\bar{m}}(k)\right]}_{=: K_{\psi}(k)} \left[(\Delta t)^{\frac{\eta}{2}} + (\Delta x)^{\eta}\right].
\end{align*}

Other cases are proved in the same way as above using Proposition \ref{I23_prop4.4,4.6}, \eqref{eq:e52} and \eqref{eq:e53}. 

What is left is to estimate the asymptotic growth rate of $K_{\bar{m}}(k)$. Set $a_k := K_{\bar{m}}(k)/\bar{C}_1$. Then, $a_k$ satisfies the recurrence formula
$$
a_0 = 0, \qquad a_{k+1} = \left[(1-\delta_k) + \delta_k \bar{C}_2\right] a_k + \delta_k.
$$
Therefore, applying Lemma \ref{lem:recurrence_ineq} with $c_0 = \bar{C}_2$ yields that
$$
K_{\bar{m}}(k) \le \bar{C}_1 k_2 k^{\max\{1, k_2(\bar{C}_2-1)\}}.
$$
This means that $K_{\bar{m}}(k)$ grows no faster than polynomially.
\end{proof}

\begin{rem}
\label{rem:log}
At first glance, one might expect $K_{\mathrm{q}}(k)$ to grow exponentially as we stated in \cite[Theorem 3.4]{I23}. However, we have succeeded in getting shaper estimates in this paper. On the other hand, by Lemma \ref{lem:recurrence_ineq}, when $\bar{C}_2>1$, $K_{\mathrm{q}}(k)$ ($\mathrm{q} = \phi, \psi, \bar{m}$) is logarithmically bounded from below as $k\to\infty$. Consequently, as can be seen from the proof of Theorem \ref{I23_thm3.5}, the coefficients $C_{\bar{m}}^{(1)}(k)$ and $C_{\bar{u}}^{(1)}(k)$ appearing in the main results, including Theorem \ref{I23_thm3.5}, also diverge logarithmically as $k\to\infty$.

We note that, in the proof of Proposition \ref{I23_thm3.4}, we focused on the case $\bar{C}_2>1$, since, as explained in Appendix \ref{sec:app}, this inequality is expected to hold in many practically relevant situations. However, sharper estimates for $K_{\mathrm{q}}(k)$ are available when $\bar{C}_2\le 1$. Indeed, if $\bar{C}_2=1$,
\begin{equation*}
K_{\bar{m}}(k) \le \bar{C}_1\sum_{\ell=0}^{k-1}\delta_\ell = \bar{C}_1\sum_{\ell=0}^{k-1}\frac{k_2}{\ell+k_1}.
\end{equation*}
Accordingly, $K_{\mathrm{q}}(k)$ grows at most logarithmically. Although this estimate does not guarantee uniform boundedness, it improves the polynomial growth bound obtained in the case $\bar{C}_2>1$ to a logarithmic one.

On the other hand, the situation changes qualitatively if $\bar{C}_2<1$. We note that Lemma \ref{lem:recurrence_ineq} actually yields a uniform bound for the sequence $\{a_k\}$ under the condition $c_0<1$. Indeed, let
$$
\bar{a}=\frac{1}{1-c_0} > 0.
$$
Then, the recurrence formula in Lemma \ref{lem:recurrence_ineq} can be rewritten as
$$
a_{k+1}-\bar{a} = [1+(c_0-1)\delta_k] (a_k-\bar{a}).
$$
Since $c_0<1$, we have $0<1+(c_0-1)\delta_k<1$. Hence, since $a_0 = 0 \le \bar{a}$, an elementary induction gives
$$
a_k\le \bar{a} = \frac{1}{1-c_0}
$$
for all $k$. In particular, applying this observation to Proposition \ref{I23_thm3.4}, where $a_k=K_{\bar{m}}(k)/\bar{C}_1$ and $c_0=\bar{C}_2$, we obtain
$$
K_{\bar{m}}(k) \le \frac{\bar{C}_1}{1-\bar{C}_2}
$$
provided that $\bar{C}_2<1$.

However, verifying the condition $\bar{C}_2<1$ is highly nontrivial in practice. In Appendix \ref{sec:app}, we give a one-dimensional example to illustrate this issue. 
\end{rem}

\medskip

Having completed the above preparations, we are ready to proceed to the following proof.

\begin{proof}[Proof of Theorem \ref{I23_thm3.5}]
\noindent We first state the proof of \eqref{eq:err-1b}. 
Recall that $U^k_{n, i}$ is defined by \eqref{eq:err-0b}, which is an approximation of $u_k(t_n, x_i)=-2 \nu \log{\phi_k(t_n, x_i)}$. 
We write as
\begin{align*}
|U^k_{n, i} - \bar{u}(t_n, x_i)| &= \left|2 \nu \log{\Phi^k_{n, i}} + \bar{u}(t_n, x_i)\right|\\
&\le 2 \nu \left|\log{\Phi^k_{n, i}} - \log{\phi_k(t_n, x_i)}\right| + \left|2\nu \log{\phi_k(t_n, x_i)} + \bar{u}(t_n, x_i)\right|
\end{align*}
and derive estimations for these two terms. 

From Propositions \ref{LP23_thm7} and \ref{LP23_thm8-2}, we get 
$$
\max_{(n, i)\in\Lambda_x} \left|2\nu \log{\phi_k(t_n, x_i)} + \bar{u}(t_n, x_i)\right| \le \|2\nu \log{\phi_k} + \bar{u}\|_{L^{\infty}(Q)} \le C \eps_k^{\frac{1}{r}} \le \frac{C}{(k+k_1)^{\frac{s}{r}}}.
$$
Thanks to the mean value theorem, for any $k, n$, and $i$, we have
$$
\left|\log{\Phi^k_{n, i}} - \log{\phi_k(t_n, x_i)}\right| = \frac{1}{P^k_{n, i}} |\Phi^k_{n, i}- \phi_k(t_n, x_i)|
$$
for some $P^k_{n, i}$ between $\phi_k(t_n, x_i)$ and $\Phi^k_{n, i}$. Since
$$
P^k_{n, i} \ge \min{\{\phi_{\min}, \Phi_{\min}\}} =: P_{\min} > 0,
$$
we deduce by Proposition \ref{I23_thm3.4}
$$
\left|\log{\Phi^k_{n, i}} - \log{\phi_k(t_n, x_i)}\right| \le \frac{1}{P_{\min}} |\Phi^k_{n, i} - \phi_k(t_n, x_i)| \le \frac{K_{\phi}(k)}{P_{\min}} \left[(\Delta t)^{\frac{\eta}{2}} + (\Delta x)^{\eta}\right]. 
$$
Therefore,
$$
\max_{n, i} \left|\log{\Phi^k_{n, i}} - \log{\phi_k(t_n, x_i)}\right| \le \frac{K_{\phi}(k)}{P_{\min}} \left[(\Delta t)^{\frac{\eta}{2}} + (\Delta x)^{\eta}\right].
$$
Summing up, we obtain \eqref{eq:err-1b}. 

\noindent We proceed to the proof of \eqref{eq:err-1a}. 
Before entering into the details, we introduce some useful notation. Set
\[
X=\{\bm{\phi}=\{\bm{\phi}_i(t)\}_{i\in\Lambda_x}:[0,T)\times \Lambda_x\to \mathbb{R}^{N_x^d}\mid 
\bm{\phi}_i(t)\in\mathcal{C}([t_n,t_{n+1})) \mbox{ for } n=0,1,\ldots,N_t-1\}
\]
and, for $\bm{\phi}\in X$, 
\[
\|\bm{\phi}\|=\left(\int_0^T \max_{i\in\Lambda}|\bm{\phi}_i(t)|^2~dt\right)^{\frac12}
\]
which defines a norm of $X$. Moreover, we introduce 
$\bar{\bm{M}}^k,
\bar{\bm{m}}^{k,\Delta t},
\bar{\bm{m}}^{\Delta t},
\bar{\bm{m}}^{k},
\bar{\bm{m}}\in X$ by setting 
\begin{align*}
\bar{\bm{M}}_i^k(t)&=\bar{M}_{n,i}^k\mbox{ for }t\in [t_n,t_{n+1}),\ n=0,1,\ldots,N_t, \\
\bar{\bm{m}}_i^{k,\Delta t}(t)&=\bar{m}_k(t_n,x_i)\mbox{ for }t\in [t_n,t_{n+1}),\ n=0,1,\ldots,N_t, \\
\bar{\bm{m}}_i^{\Delta t}(t)&=\bar{m}(t_n,x_i)\mbox{ for }t\in [t_n,t_{n+1}),\ n=0,1,\ldots,N_t, \\
\bar{\bm{m}}_i^{k}(t)&=\bar{m}_k(t,x_i)\mbox{ for }t\in [0,T) \\
\bar{\bm{m}}_i(t)&=\bar{m}(t,x_i)\mbox{ for }t\in [0,T).
\end{align*}

With this notation, we can estimate $I(\bar{M}^k,\bar{m})$ as 
\begin{align}
I(\bar{M}^k, \bar{m}) &= \|\bar{\bm{M}}^k-\bar{\bm{m}}^{\Delta t}\| \nonumber \\
&\le 
\|\bar{\bm{M}}^k-\bar{\bm{m}}^{k,\Delta t}\|
+\|\bar{\bm{m}}^{k,\Delta t}-\bar{\bm{m}}^{k}\|
+\|\bar{\bm{m}}^{k}-\bar{\bm{m}}\|
+\|\bar{\bm{m}}-\bar{\bm{m}}^{\Delta t}\|.\label{eq:proof-I}
\end{align}
By virtue of Proposition \ref{I23_thm3.4}, we have
\begin{align*}
\|\bar{\bm{M}}^k-\bar{\bm{m}}^{k,\Delta t}\|
& = \left(\sum_{n=0}^{N_t-1}\int_{t_n}^{t_{n+1}} \max_{i\in \Lambda_x}|M_{n,i}^k-\bar{m}_k(t_n,x_i)|^2~dt\right)^{\frac12}\\
&\le \sqrt{T} E(\bar{M}^k,\bar{m}_k) \\
& \le \sqrt{T} K_{\bar{m}}(k) \left[(\Delta t)^{\frac{\eta}{2}} + (\Delta x)^{\eta}\right].
\end{align*}
In view of Propositions \ref{LP23_thm7} and \ref{LP23_thm8-2}, 
\[
 \|\bar{\bm{m}}^{k}-\bar{\bm{m}}\|\le \|\bar{m}_k-\bar{m}\|_{L^2(0,T;L^\infty(\T^d))}\le 
 C\eps^{\frac{1}{r}}\le \frac{C}{(k+k_1)^{\frac{s}{r}}}.
\]

To estimate the second and fourth terms of the right-hand side of \eqref{eq:proof-I}, we use the Lipschitz continuity of $\bar{m}$ and $\bar{m}_k$. Indeed, by Propositions \ref{B21_thm1} and \ref{regularity_phi_psi} (and the choice of $\bar{m}_0$), we have
\[
\|\bar{m}\|_{\mathcal{C}^{1+\eta/2,2+\eta}(Q)}\le C,\quad 
\|\bar{m}_k\|_{\mathcal{C}^{1+\eta'/2,2+\eta'}(Q)}\le C
\]
In particular, both functions are Lipschitz continuous in $t$, and their Lipschitz constants are bounded by the above norms.

The second term is estimated as
\begin{align*}
\|\bar{\bm{m}}^{k,\Delta t}-\bar{\bm{m}}^{k}\| 
& = \left(\sum_{n=0}^{N_t-1}\int_{t_n}^{t_{n+1}} \max_{i\in \Lambda_x}|\bar{m}_k(t_n,x_i)-\bar{m}_k(t,x_i)|^2~dt\right)^{\frac12}\\
& = \left(\sum_{n=0}^{N_t-1}\int_{t_n}^{t_{n+1}} C|t_n-t|^2~dt\right)^{\frac12}\le C \sqrt{T} \Delta t.
\end{align*}
In exactly the same way, 
\[
\|\bar{\bm{m}}-\bar{\bm{m}}^{\Delta t}\|\le C\sqrt{T}\Delta t. 
\]

Summing up these estimates, we complete the proof of \eqref{eq:err-1a}. 

\end{proof}

\begin{proof}[Proof of Theorems \ref{I23_thm3.5a} and \ref{I23_thm3.5b}]
They follow the same lines as that of Theorem~\ref{I23_thm3.5}, with the only modification being that the inequalities taken from Propositions~\ref{I23_prop4.4,4.6} and Proposition~\ref{I23_thm3.4} are replaced by the corresponding ones.
\end{proof}


\section{Numerical Experiments}
\label{sec:numerical_experiment}

In this section, we provide some numerical examples and verify the validity of Theorem \ref{I23_thm3.5}. Throughout this section, numerical computations are carried out in Python on a laptop with an AMD Ryzen AI 9 365 processor (10 cores) and 32GB RAM.

\subsection{One-dimensional Case}

\begin{example}
\label{ex:1D}
Letting $d=1$, we set as: 
\begin{alignat*}{2}
\mbox{$T, \nu$ and Hamiltonian:} &&\quad& T = 0.1, \quad \nu = 0.01, \quad H(t, x, p) = \frac{1}{2}|p|^2;\\
\mbox{terminal condition:} &&&g(x) = -\frac{1}{2 \pi} \cos{(2 \pi x)}; \\
\mbox{initial condition:}&&&m_0(x) = \frac{1}{\sqrt{2 \pi \sigma^2}} \exp{\left(-\frac{(x-1/2)^2}{2 \sigma^2}\right)} \quad (\sigma = 0.1);\\
\mbox{coupling term:}&&&f(t,x, m) = \left(x-\frac{1}{2}\right)^2 + 4 \min{\{m(t,x), 5\}}.
\end{alignat*}
\end{example}

In the setting of Example \ref{ex:1D}, the solution of \eqref{MFG_eq} exhibits the following behavior. 
At the initial time \( t = 0 \), the density distribution is given by a normal distribution with mean \( 1/2 \) and variance \( \sigma^2 \).
At the terminal time $t=T$, the prescribed terminal cost $g$ determines the control input as
$$
-\del_p H(T,x,g'(x))=-\sin(2\pi x),
$$
and consequently the density distribution evolves so as to disperse toward the endpoints \( x=0 \) and \( x=1 \) of the interval \([0,1]\).
As for the coupling term, the first term \( (x - 1/2)^2 \) promotes concentration of the density around \( x = 1/2 \).
The second term \( \min\{m(x), 5\} \) serves to mitigate congestion effects. 
The time evolution of the density distribution $m$, the value function $u$, and the control input $v = -\del_p H(t, x, \del_x u(t, x))$ is illustrated in Figure~\ref{figure1}.

\begin{figure}[ht]
\centering
\includegraphics[height=4cm]{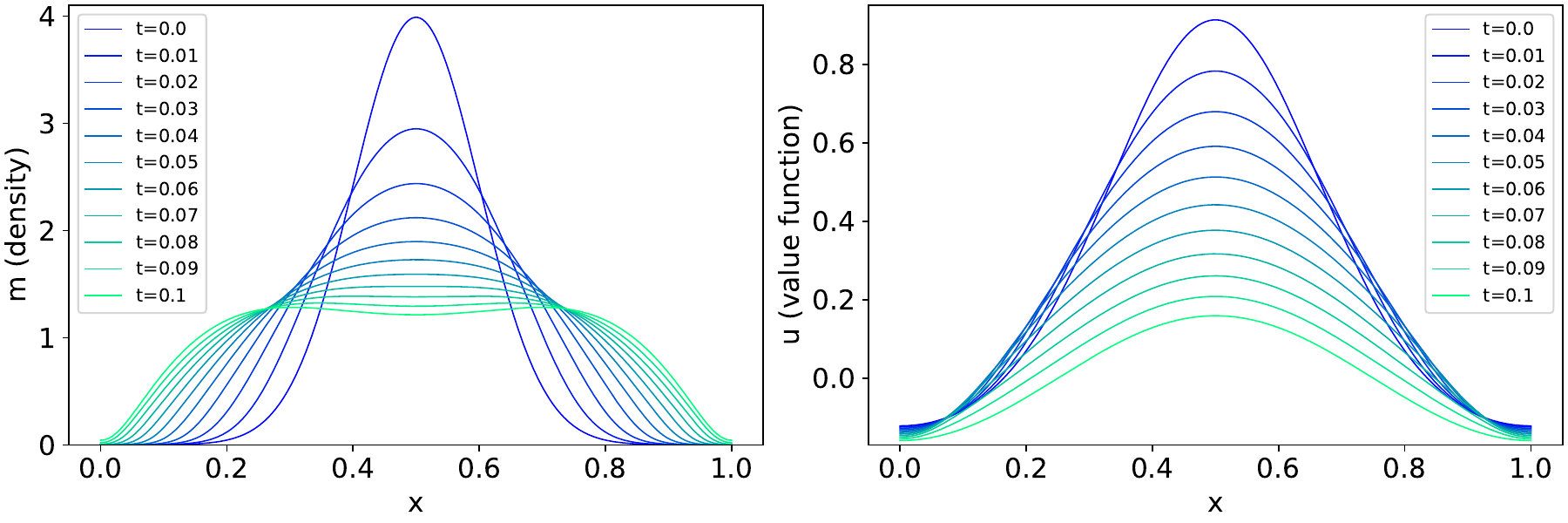}

\vspace{1em}

\includegraphics[height=4cm]{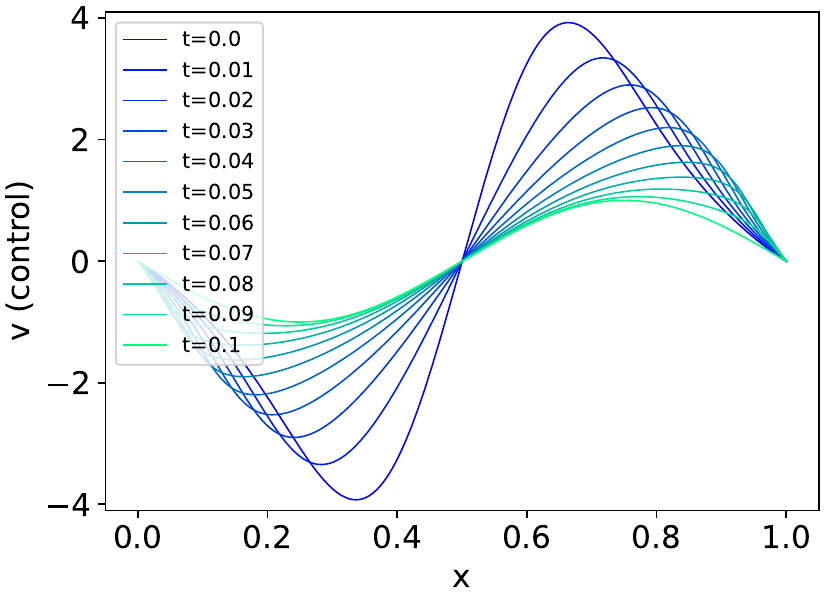}

\vspace{-8pt}
\caption{Upper left panel: density distribution; upper right panel: value function; lower panel: control input (Example~\ref{ex:1D}).}
\label{figure1}
\end{figure}

We now seek to validate, by means of numerical experiments, the explicit convergence-rate estimates derived in Theorem \ref{I23_thm3.5}. Since an exact solution to the MFG system \eqref{MFG_eq} is not available, we employ as a reference solution a numerically computed solution obtained using a sufficiently fine mesh and a reasonably large number of iterations, which serves as a substitute for the exact solution. 
More precisely, the reference solution $\bar{m}_{\textup{ref}} \approx \bar{m}$ is defined as the numerical solution $\bar{M}^k$ computed with
$N_t = 5760$, $N_x = 1200$, $\delta_k = 10/(k+10)$, and a sufficiently large iteration number $k$.

We first choose $\delta_k = 1/(k+1)$ and carry out the iteration up to $k = 1000$. Under this setting, we expect the following relation to hold according to Theorem \ref{I23_thm3.5}; 
$$
I(\bar{M}^k, \bar{m}_{\mathrm{ref}}) \approx I(\bar{M}^k, \bar{m}) \le C_{\bar{m}}^{(1)}(k) \left[(\Delta t)^{\frac{\eta}{2}} + (\Delta x)^{\eta}\right] + \frac{C_{\bar{m}}^{(2)}}{\sqrt{k+1}} \approx C_{\bar{m}}^{(1)}(k) \left[(\Delta t)^{\frac{\eta}{2}} + (\Delta x)^{\eta}\right].
$$

\begin{figure}[ht]
\centering
{\includegraphics[height=4.5cm]{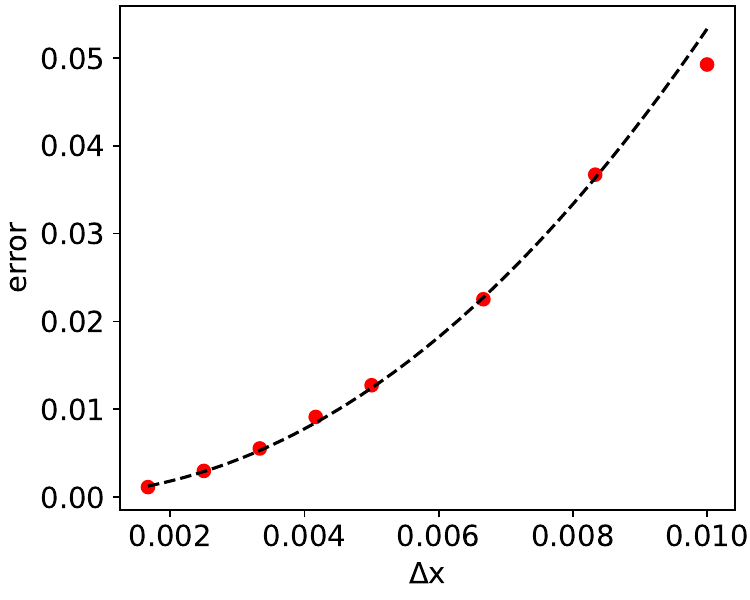}}
\vspace{-8pt}
\caption{Values of $I(\bar{M}^k, \bar{m}_{\mathrm{ref}})$ for different values of $\Delta x$ (red points) and the corresponding regression curve (Example~\ref{ex:1D}).}
\label{figure2}
\end{figure}

We compute the numerical solutions $\bar{M}^k$ for various choices of $\Delta t$ and $\Delta x$, where the numbers of time and space subdivisions are selected so that $2\nu\Delta t/(\Delta x)^2x = 1/2$.
Figure~\ref{figure2} shows that the observed convergence rate is approximately
\[
I(\bar{M}^k, \bar{m}_{\mathrm{ref}}) = \mathcal{O}\bigl((\Delta x)^{2.103}\bigr).
\]
(The convergence rate was estimated by linear regression on $\log(\Delta x)$ and $\log I(\bar{M}^k, \bar{m}_{\mathrm{ref}})$ using \texttt{LinearRegression} from the \texttt{sklearn.linear\_model} library.) 
While the theoretical analysis yields $\eta \in (0,1)$, this numerical result indicates a higher order of convergence.
Indeed, as stated in Theorems~\ref{I23_thm3.5a} and \ref{I23_thm3.5b}, the above result also suggests higher regularity of $\phi_k$ and $\psi_k$.

Next, the step-size is updated according to $\delta_k = k_2/(k + k_1)$, and we investigate how the convergence behavior depends on the choice of the parameters $k_1$ and $k_2$.
For this purpose, we fix $N_t = 1000$ and $N_x = 500$, and define the reference solution $\bar{m} = \bar{m}_{\textup{ref}}$ as the numerical solution obtained with $\delta_k = 10/(k+10)$ and a sufficiently large number of iterations.
We then compute the numerical solution $\bar{M}^k$ using the same values of $N_t$ and $N_x$, but with the step-size $\delta_k = k_2/(k + k_1)$.

According to Theorem~\ref{I23_thm3.5}, we expect
$$
I(\bar{M}^k,\bar{m}_{\mathrm{ref}}) \approx I(\bar{M}^k,\bar{m}) \leq C_{\bar{m}}^{(1)}(k)\left[(\Delta t)^{\frac{\eta}{2}}+(\Delta x)^\eta\right]+\frac{C_{\bar{m}}^{(2)}}{(k+k_1)^{\frac{k_2}{2}}} \approx \frac{C_{\bar{m}}^{(2)}}{(k+k_1)^{\frac{k_2}{2}}},
$$
provided that the discretization-error term is sufficiently small. We expect a similar estimate for $E(U^k,\bar{u}_{\mathrm{ref}})$, where $\bar{u}_{\mathrm{ref}}$ is a reference solution computed in the same manner as $\bar{m}_{\mathrm{ref}}$.

Therefore, by computing the values of $I(\bar{M}^k, \bar{m}_{\mathrm{ref}})$ for various choices of the step-size $\delta_k = k_2/(k + k_1)$ and iteration number $k$, we numerically estimate the convergence order $k_2/2$.

Figure~\ref{figure5} displays the values of $I(\bar{M}^k, \bar{m}_{\mathrm{ref}})$ for various choices of the step size $\delta_k = k_2/(k + k_1)$.
The approximate convergence orders for each case are reported in Tables~\ref{table1-1} and \ref{table1-2}.
From Tables~\ref{table1-1} and \ref{table1-2}, we observe that
\[
I(\bar{M}^k, \bar{m}_{\mathrm{ref}}) \approx \mathcal{O}\bigl((k + k_1)^{-k_2}\bigr), \qquad E(U^k, \bar{u}_{\mathrm{ref}}) \approx \mathcal{O}\bigl((k + k_1)^{-k_2}\bigr), \qquad 
\]
which are consistent with the theoretical results. In addition, the decay order of $I(\bar{M}^k ,\bar{m}_{\mathrm{ref}})$ mainly depends on $k_2$, while the influence of $k_1$ appears to be relatively small. Accordingly, for the same value of $k_1$, a larger value of $k_2$ leads to faster decay of the error. However, the observed convergence orders are twice as high as the theoretically expected rate
$\mathcal{O}\bigl((k + k_1)^{-k_2/2}\bigr)$.

This discrepancy suggests that there may be room for improvement in the proof of Theorem~\ref{LP23_thm8-2}.
Indeed, the arguments presented in \cite{NS26} for this theorem rely on several simplifying assumptions introduced for computational convenience.
A more refined and rigorous analysis may lead to an estimate that better aligns with the numerical results.

\begin{figure}[ht]
\centering
{\includegraphics[height=8cm]{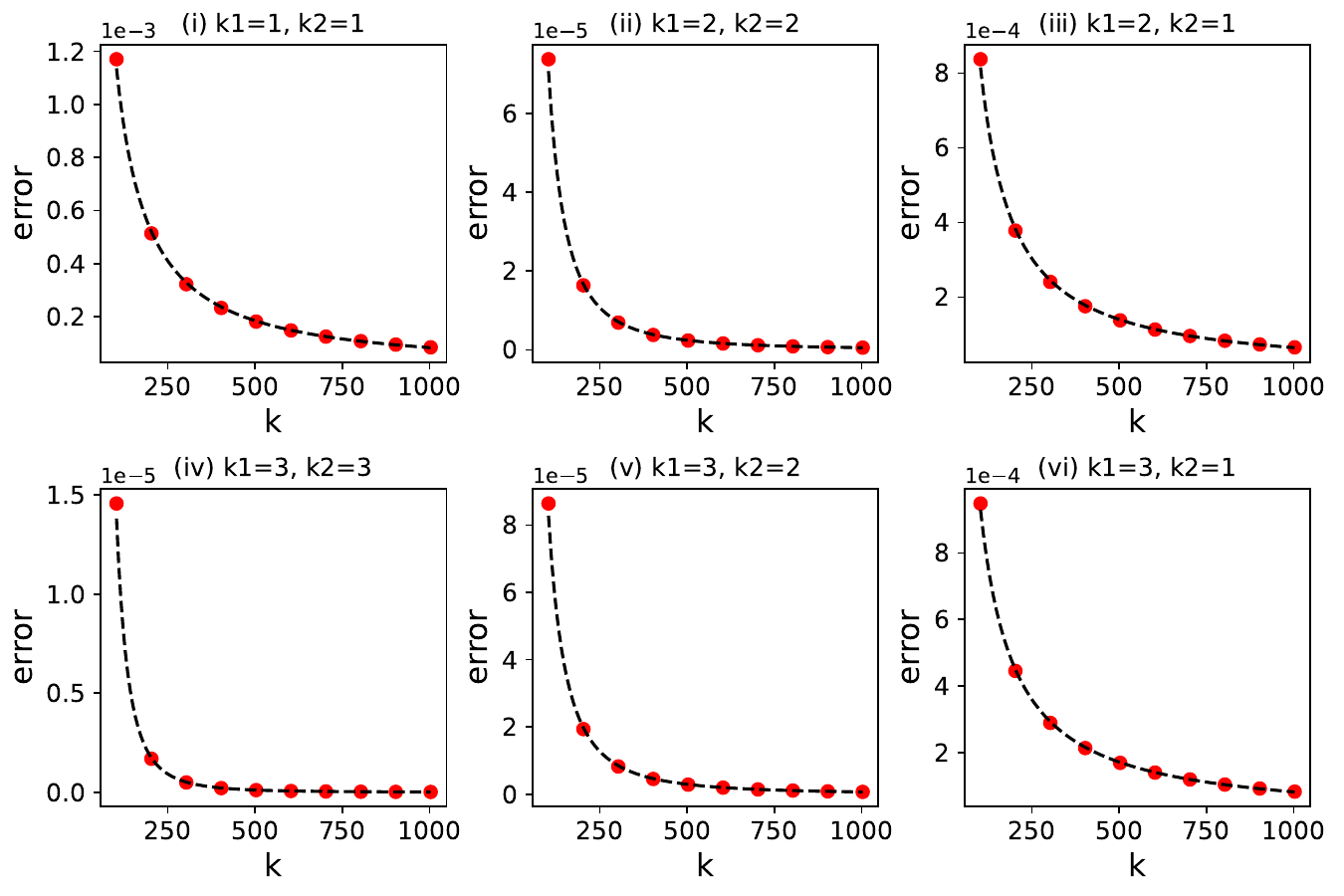}}
\caption{Computed values of $I(\bar{M}^k, \bar{m}_{\mathrm{ref}})$ for different step sizes $\delta_k = k_2/(k + k_1)$ and iteration indices $k$ (red points), along with the corresponding regression curves (Example~\ref{ex:1D}).}
\label{figure5}
\end{figure}

\begin{table}[th]
\caption{Estimated convergence orders of $I(\bar{M}^k, \bar{m}_{\mathrm{ref}})$ for each step-size (Example~\ref{ex:1D}).}
\centering
\label{table1-1}
\begin{tabular}{c||c|c|c} \hline
 $\delta_k=\frac{k_2}{k+k_1}$ & $k_1=1, k_2=1$ & $k_1=2, k_2=2$ & $k_1=2, k_2=1$ \\ \hline
 $I(\bar{M}^k, \bar{m}_{\mathrm{ref}})$ & $\mathcal{O}\left((k+1)^{-1.142}\right)$ & $\mathcal{O}\left((k+2)^{-2.121}\right)$ & $\mathcal{O}\left((k+2)^{-1.110}\right)$ \\ \hline
\end{tabular}

\begin{tabular}{c||c|c|c} \hline
 $\delta_k=\frac{k_2}{k+k_1}$ & $k_1=3, k_2=3$ & $k_1=3, k_2=2$ & $k_1=3, k_2=1$\\ \hline
 $I(\bar{M}^k, \bar{m}_{\mathrm{ref}})$ & $\mathcal{O}\left((k+3)^{-3.056}\right)$ & $\mathcal{O}\left((k+3)^{-2.111}\right)$ & $\mathcal{O}\left((k+3)^{-1.071}\right)$ \\ \hline
\end{tabular}
\end{table}

\begin{table}[th]
\caption{Estimated convergence orders of $E(U^k, \bar{u}_{\mathrm{ref}})$ for each step-size (Example~\ref{ex:1D}).}
\centering
\label{table1-2}
\begin{tabular}{c||c|c|c} \hline
 $\delta_k=\frac{k_2}{k+k_1}$ & $k_1=1, k_2=1$ & $k_1=2, k_2=2$ & $k_1=2, k_2=1$ \\ \hline
 $E(U^k, \bar{u}_{\mathrm{ref}})$ & $\mathcal{O}\left((k+1)^{-1.085}\right)$ & $\mathcal{O}\left((k+2)^{-2.307}\right)$ & $\mathcal{O}\left((k+2)^{-1.190}\right)$ \\ \hline
\end{tabular}

\begin{tabular}{c||c|c|c} \hline
 $\delta_k=\frac{k_2}{k+k_1}$ & $k_1=3, k_2=3$ & $k_1=3, k_2=2$ & $k_1=3, k_2=1$\\ \hline
 $E(U^k, \bar{u}_{\mathrm{ref}})$ & $\mathcal{O}\left((k+3)^{-3.259}\right)$ & $\mathcal{O}\left((k+3)^{-2.345}\right)$ & $\mathcal{O}\left((k+3)^{-1.197}\right)$ \\ \hline
\end{tabular}
\end{table}

We briefly comment here on a practical stopping criterion. In the present numerical experiments, our main purpose is to validate the main theorem. Therefore, as explained above, we fix the number of iterations at $k=1000$. Nevertheless, Figure \ref{figure5} also shows that the error decreases more slowly as the iteration proceeds. This suggests that, in practical computations, the iteration may be terminated once successive iterates differ only slightly. Since the error plotted in the figure is defined relative to the reference solution $\bar{m}$, it is not available in an actual computation. Instead, one may monitor the update size
$$
e_k := \|\bar{M}^{k+1} - \bar{M}^k\|_{\infty}
$$
and terminate the iteration once $e_k$ falls below a prescribed tolerance. For example, \cite{G12, I23} use the stopping criterion $e_k < 10^{-7}$.

Finally, we examine how the decay of the error $I(\bar{M}^k,\bar{m})$ changes as $k_0$ increases in the step-size $\delta_k=k_0/(k+k_0)$. The left panel of Figure \ref{figure_ex1} shows the corresponding error decay. For finite values of $k_0$, we observe that the onset of the decay is delayed as $k_0$ increases. This suggests that choosing $\delta_k$ too close to $1$ may slow down the convergence.

Motivated by this observation, we also plot the case $\delta_k=1$, which is formally represented by $k_0=\infty$ in Figure \ref{figure_ex1}. In this case, the error does not decrease. The right panel of Figure \ref{figure_ex1} further shows the error decay for $\delta_k=1$ up to $k=3000$. The error oscillates around approximately $1.8$, and no convergence behavior is observed even after this larger number of iterations. We note that the convergence of the proposed scheme \eqref{eq:dGCG} for $\delta_k=1$ is not covered by our main result.

As can be seen from \eqref{eq:dGCG}, the case $\delta_k=1$ corresponds to using the newly computed $M^k$ directly as the next iterate. This result indicates that such a naive update rule does not yield convergence in this example. This highlights the importance of using the GCG method with an appropriate step-size rule, rather than relying on the naive update $\bar{M}^{k+1}=M^k$.

A closer observation of Figure \ref{figure_ex1} provides some further insight into the effect of the step-size. For $k_0=2, 5, 10, 50$, and $100$, represented by the yellow, green, red, purple, and brown curves, respectively, the error begins to decrease rapidly at approximately $k=20, 50,100,500$, and $1000$. At these iteration numbers, the corresponding step sizes satisfy
$$
\delta_k=\frac{k_0}{k+k_0}\approx\frac{1}{11},
$$
and are therefore all close to $0.1$. This observation suggests that, for the present example, particularly rapid error decay may occur when $\delta_k$ is around $0.1$.

As $k_0$ increases, more iterations are required for $\delta_k$ to reach this range. At the same time, the range of iteration numbers over which $\delta_k$ remains close to $0.1$ becomes wider. This appears to be consistent with the behavior observed in Figure~\ref{figure_ex1}: for larger $k_0$, the onset of the decay is delayed, but once the decay begins, it continues over a longer range of iterations.

This observation suggests a possible strategy for accelerating convergence. If a favorable step-size $\bar{\delta}\in(0,1)$ could be found in advance, the parameters in the step-size rule could be chosen so that $\delta_k = k_2/(k+k_1)$ remains close to $\bar{\delta}$ over a suitable range of iterations. However, based on our experience with related numerical experiments, the favorable value $\bar{\delta}$ appears to depend on the particular problem. Determining such a $\bar{\delta}$ a priori is not straightforward. We therefore do not pursue this issue further here.

\begin{figure}[ht]
\centering
\begin{minipage}{0.48\textwidth}
\centering
\includegraphics[height=5cm]{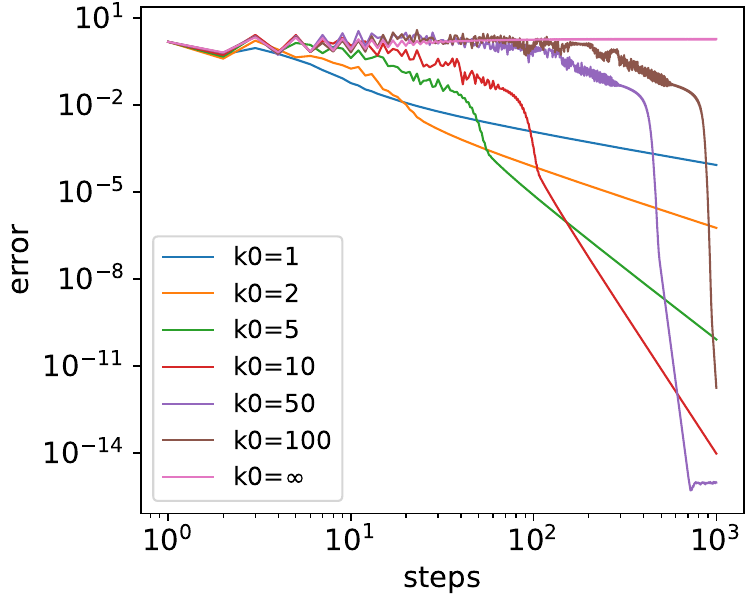}
\end{minipage}
\hfill
\begin{minipage}{0.48\textwidth}
\centering
\includegraphics[height=5cm]{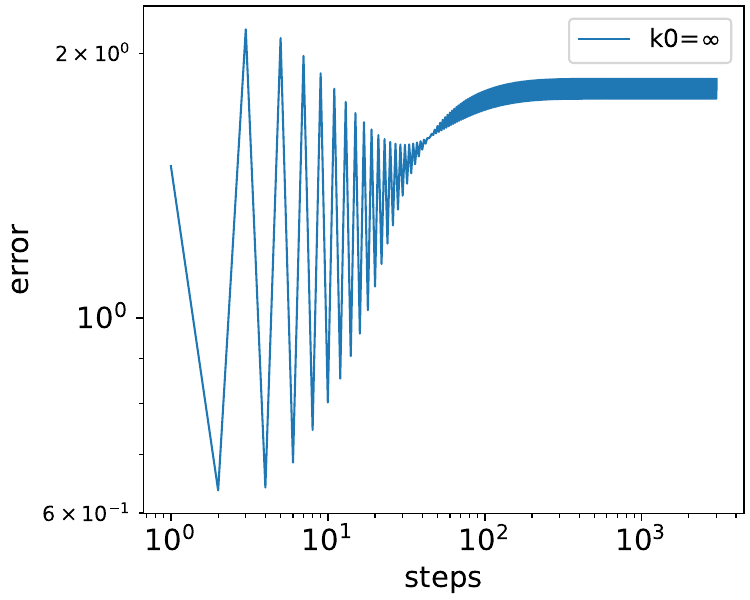}
\end{minipage}
\vspace{-8pt}
\caption{Decay of the error $I(\bar{M}^k,\bar{m}_{\mathrm{ref}})$ for different choices of the step size $\delta_k=k_0/(k+k_0)$. Left panel: comparison of several choices of $k_0$; right panel: long-time behavior for $\delta_k=1$ ($k_0 = \infty$). (Example \ref{ex:1D}.)}
\label{figure_ex1}
\end{figure}

We now consider the case in which a nontrivial advection effect appears in the Hamiltonian as given in Example \ref{ex:1Db}.

\begin{example}
\label{ex:1Db}
The Hamiltonian is $\displaystyle{H(t,x,p) = \frac{|p|^2}{2} - p}$. 
The remaining is the same as in Example~\ref{ex:1D}.
\end{example}

In the setting of Example \ref{ex:1Db}, the convergence rate parameter satisfies at most $\eta = 1$.
This limitation in the error estimate stems from the discretization error associated with the upwind difference approximation. 
We verify this limitation numerically by modifying only the Hamiltonian in the current setting as above and performing the same numerical experiment as before. The time evolution of the density distribution, the value function, and the control input is shown in Figure~\ref{figure3}.

\begin{figure}[ht]
\centering
\includegraphics[height=4cm]{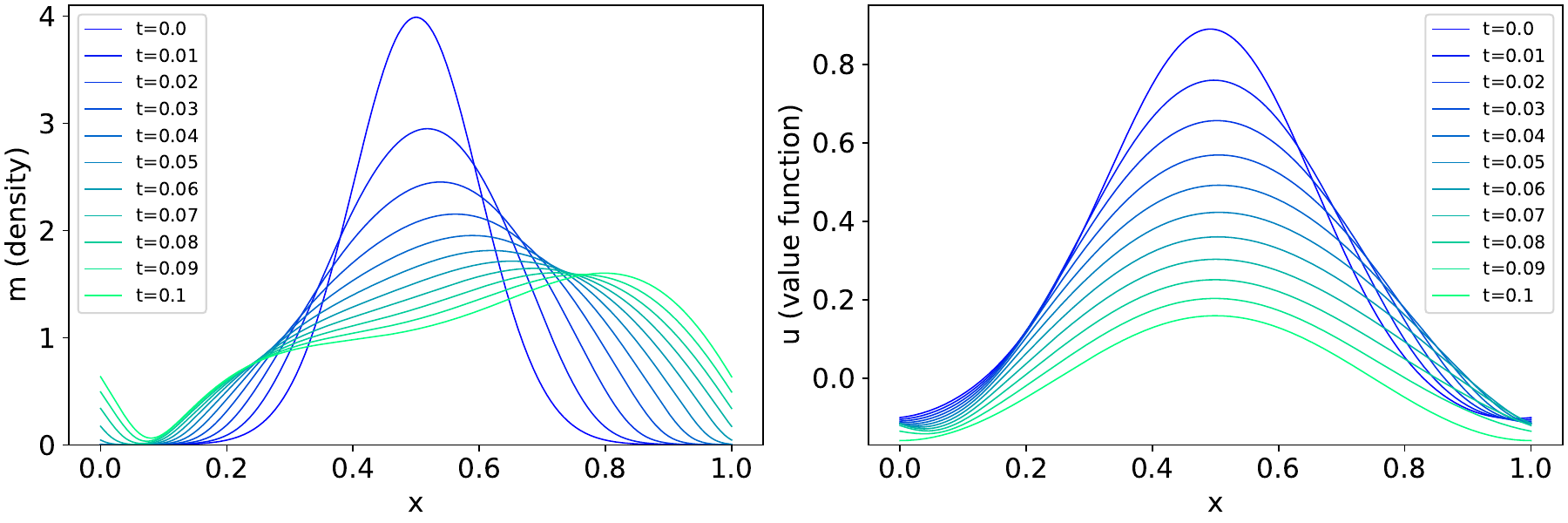}

\vspace{1em}

\includegraphics[height=4cm]{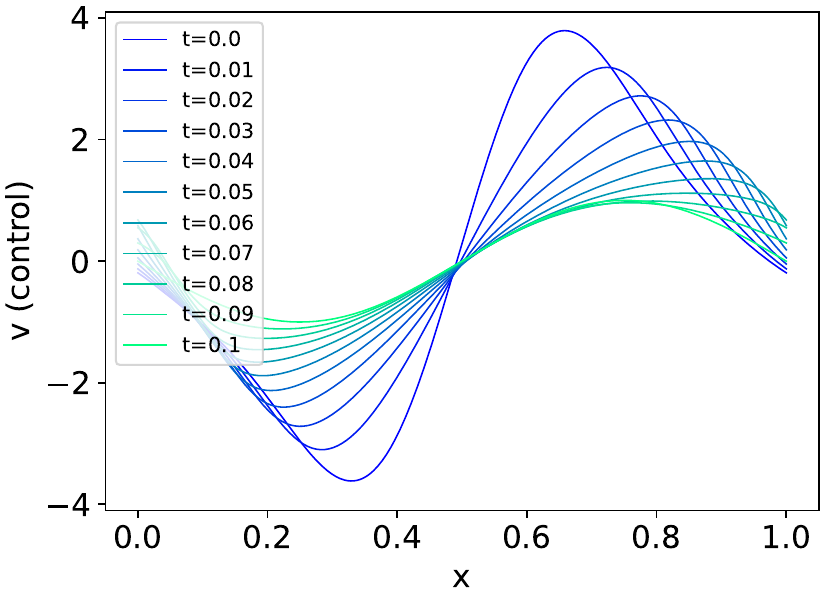}
\vspace{-8pt}
\caption{Upper left panel: density distribution; upper right panel: value function; lower panel: control input (Example~\ref{ex:1Db}).}
\label{figure3}
\end{figure}

\begin{figure}[ht]
\centering
{\includegraphics[height=4.5cm]{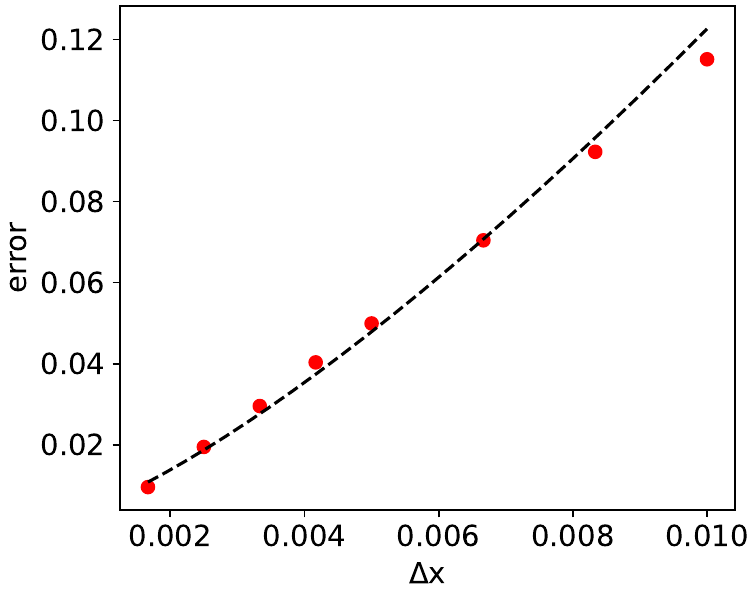}}
\vspace{-8pt}
\caption{
Values of $I(\bar{M}^k, \bar{m}_{\mathrm{ref}})$ for different values of $\Delta x$ (red points) and the corresponding regression curve (Example~\ref{ex:1Db}).}
\label{figure4}
\end{figure}

Figure~\ref{figure4} shows that the observed convergence rate is approximately
\[
I(\bar{M}^k, \bar{m}_{\mathrm{ref}}) = \mathcal{O}\bigl((\Delta x)^{1.356}\bigr),
\]
which supports the above explanation.
For comparison, similar numerical experiments reported in \cite[Section~5]{I23} yielded convergence rates of approximately
$\mathcal{O}\bigl((\Delta x)^{1.04}\bigr)$ and $\mathcal{O}\bigl((\Delta x)^{1.33}\bigr)$.
Note, however, that in \cite{I23} the norm
\[
\|\bar{M}^k - \bar{m}\|_{\infty}
= \sup_{n,i} \lvert \bar{M}^k_{n,i} - \bar{m}(t_n, x_i) \rvert
\]
is used instead of $I(\bar{M}^k, \bar{m})$.

\subsection{Two-dimensional Case}

\begin{example}
\label{ex:2D}
Letting $d=2$, we set
\begin{alignat*}{2}
\text{$T, \nu$ and Hamiltonian}: &&\quad& \qquad T = 0.1, \quad \nu = 0.01, \quad H(t, x, p) = \frac{1}{2}|p|^2;\\
\text{terminal condition}: &&& \qquad g(x) = -\frac{1}{2 \pi} \left[\cos{(2 \pi x_1)}+\cos{(2 \pi x_2)}\right];\\
\text{initial condition}: &&& \qquad m_0(x) = \frac{1}{2 \pi \sigma^2} \exp{\left(-\frac{|x-x_c|^2}{2 \sigma^2}\right)} \quad (\sigma = 0.25);\\
\text{coupling term}: &&& \qquad f(t,x, m) = |x-x_c|^2 + 4 \min{\{m(t,x), 5\}},
\end{alignat*}
where $x = (x_1, x_2)\in \T^2$ and $x_c = (1/2, 1/2)$. 
\end{example}

In the setting of Example \ref{ex:1D}, the time evolution of the density distribution is shown in Figure \ref{figure6}, and the control input is shown in Figure \ref{figure7}. (Since the control remains almost unchanged due to the short time scale, we plot it only at time $t = T$.)

\begin{figure}[ht]
\centering
{\includegraphics[height=7cm]{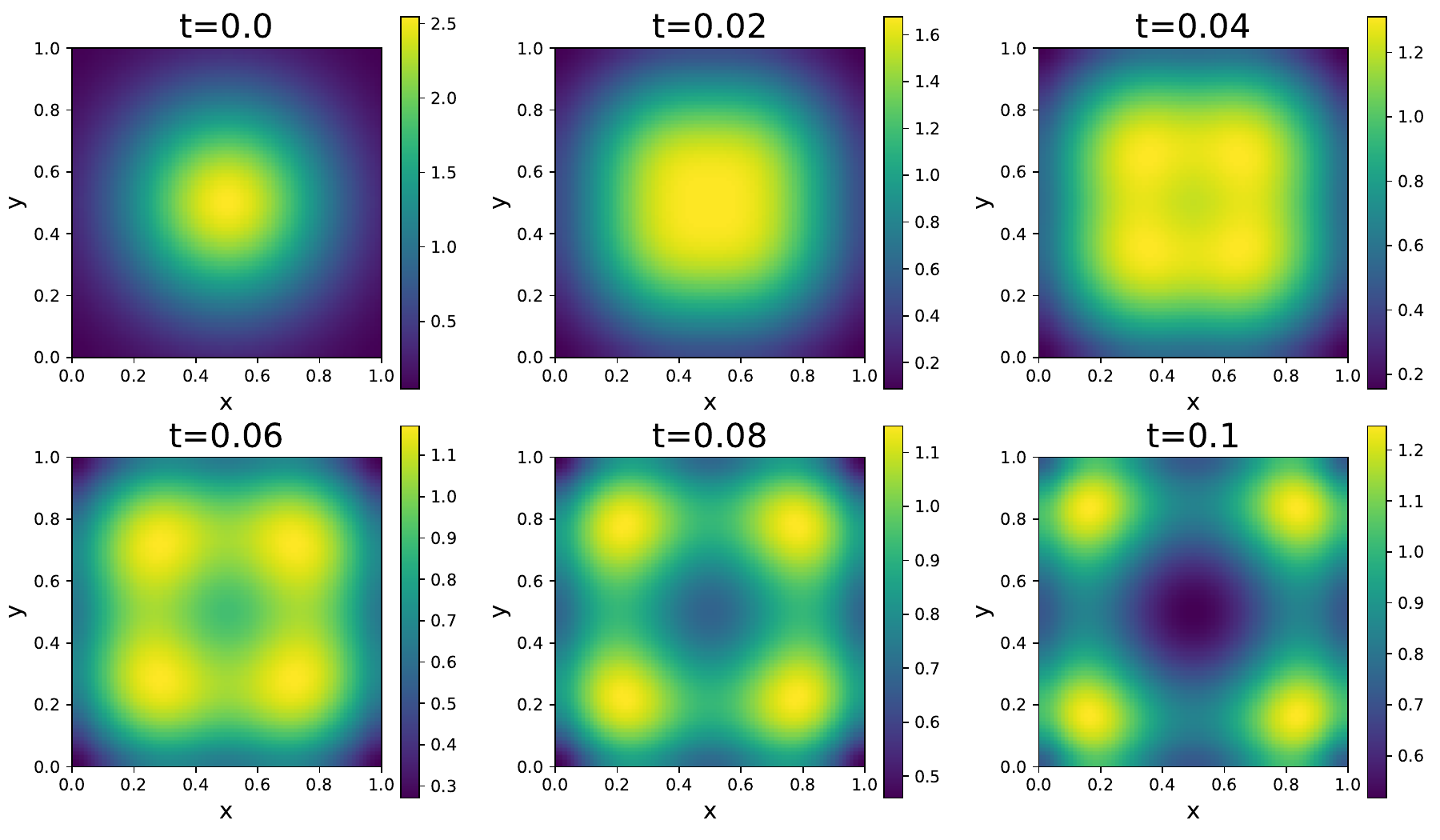}}
\caption{Density distribution $m$ (Example \ref{ex:2D}).}
\label{figure6}
\end{figure}

\begin{figure}[ht]
\centering
{\includegraphics[height=5cm]{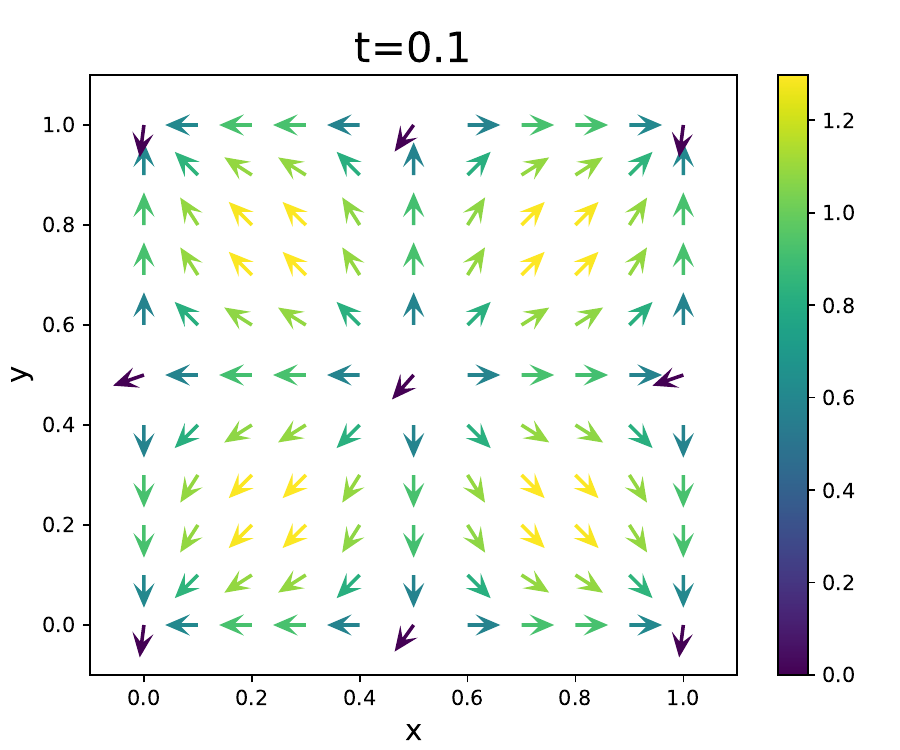}}
\caption{Vector field of the control input at time $t = T$ (the scale represents the norm of the vector) (Example \ref{ex:2D}).}
\label{figure7}
\end{figure}

First, as in the one-dimensional case, we examine the convergence order with respect to the discretization parameters.
We define the reference solution $\bar{m}_{\textup{ref}} \approx \bar{m}$ as the numerical solution obtained with
$N_t = 400$, $N_x = 200$, $\delta_k = 10/(k+10)$, and a sufficiently large number of iterations.
We then compute the numerical solution $\bar{M}^k$ using $\delta_k = 1/(k+1)$ and $k = 1000$.
As explained in the one-dimensional case, we expect that
\[
I(\bar{M}^k, \bar{m}_{\mathrm{ref}}) \approx I(\bar{M}^k, \bar{m})
\approx
C_{\bar{m}}^{(1)}(k)
\left[
(\Delta t)^{\frac{\eta}{2}} + (\Delta x)^{\eta}
\right].
\]

\begin{figure}[ht]
\centering
{\includegraphics[height=4.5cm]{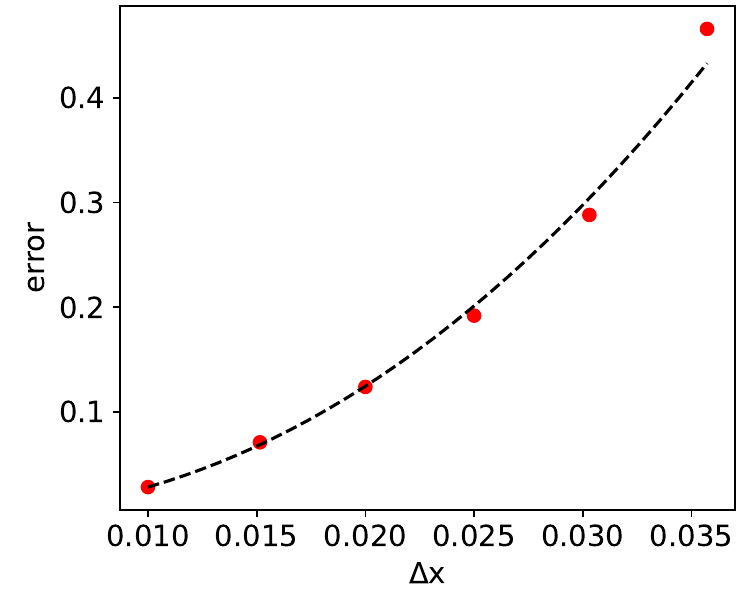}}
\caption{Values of $I(\bar{M}^k, \bar{m}_{\mathrm{ref}})$ for different values of $\Delta x$ (red points) and the corresponding regression curve (Example \ref{ex:2D}).}
\label{figure8}
\end{figure}

We compute the numerical solutions $\bar{M}^k$ for various values of $\Delta t$ and $\Delta x$, where the numbers of time and space subdivisions are chosen so that $4\nu\Delta t/(\Delta x)^2 \approx 2/5$. 
For the comparison with those coarser-mesh numerical solutions, a direct pointwise comparison would require all coarse-grid points to be included in the reference grid. This nesting requirement would force the reference grid to be very fine when several coarse mesh sizes are tested. While the corresponding increase in computational cost is still acceptable in the one-dimension case, it becomes significant in the present two-dimensional experiments, where the number of grid points grows rapidly under mesh refinement. Therefore, in the error computation, we use the linear interpolation of $\bar{m}_{\mathrm{ref}}$.
(We use \texttt{RegularGridInterpolator} from the \texttt{scipy.interpolate} library.) Consequently, the reported errors are measured with respect to the interpolated reference solution.
Figure~\ref{figure8} shows that the observed convergence rate is
\[
I(\bar{M}^k, \bar{m}_{\mathrm{ref}}) = \mathcal{O}\bigl((\Delta x)^{2.151}\bigr).
\]
As in the one-dimensional case, this result suggests higher regularity of $\phi_k$ and $\psi_k$.

Next, we examine the convergence order with respect to the iteration index $k$.
We fix $N_t = 64$ and $N_x = 80$, and define the reference solution $\bar{m}$ as the numerical solution obtained with $\delta_k = 10/(k+10)$ and a sufficiently large number of iterations.
We then compute the numerical solution $\bar{M}^k$ using the same values of $N_t$ and $N_x$, but with the step size $\delta_k = k_2/(k + k_1)$.
In this case, we expect that
\[
I(\bar{M}^k, \bar{m}_{\mathrm{ref}}) \approx I(\bar{M}^k, \bar{m})
\approx
\frac{C_{\bar{m}}^{(2)}}{(k + k_1)^{\frac{s}{r}}},
\]
where $r > 2$ is a constant and $s$ satisfies $s < \min\{k_2, 1/\rho\}$, with $
\rho = 1 - 2/(r(r-1))$.

\begin{figure}[ht]
\centering
{\includegraphics[height=8cm]{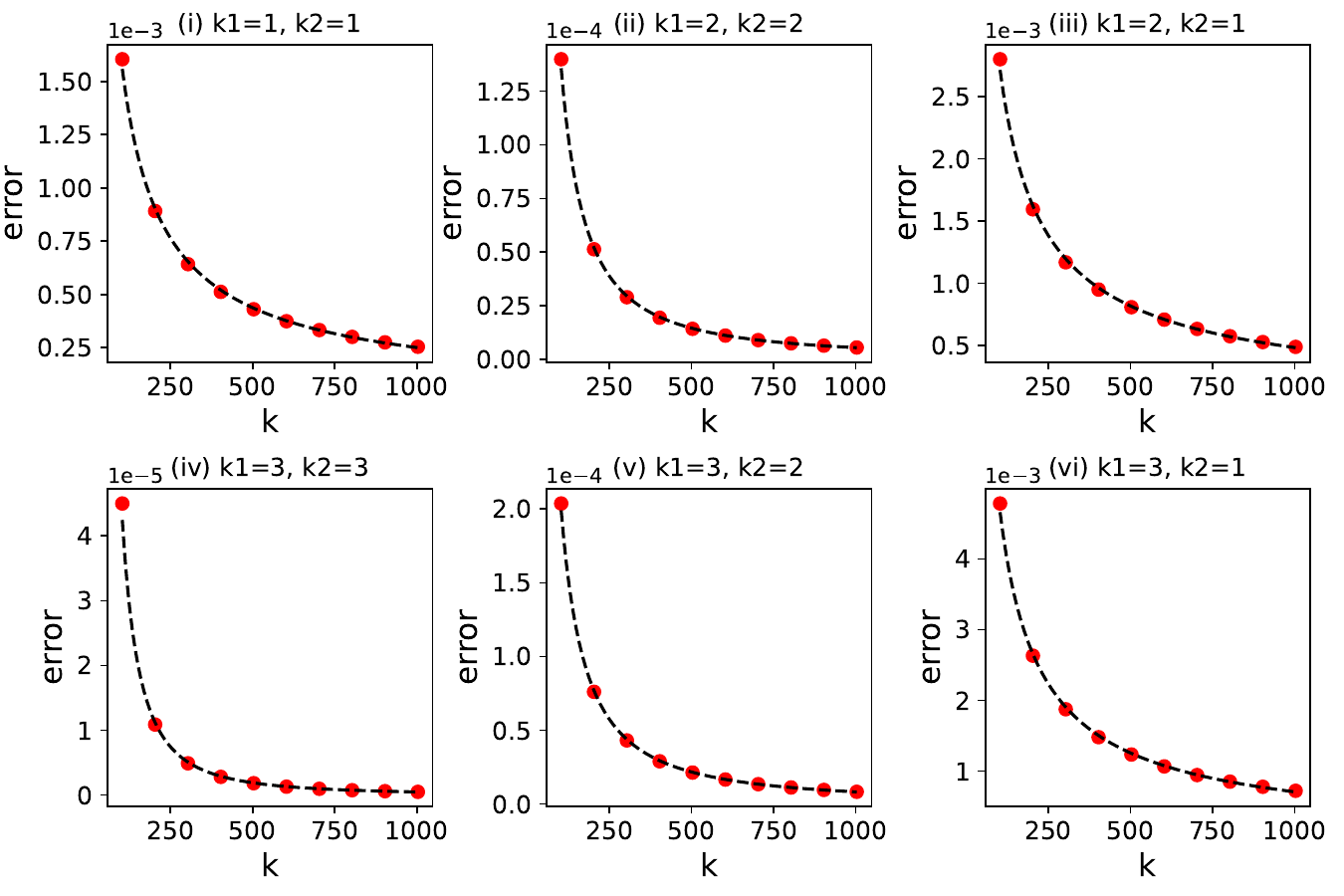}}
\caption{Computed values of $I(\bar{M}^k, \bar{m_{\mathrm{ref}}})$ for different step sizes $\delta_k = k_2/(k + k_1)$ and iteration indices $k$ (red points), along with the corresponding regression curves (Example \ref{ex:2D}).}
\label{figure9}
\end{figure}

Figure~\ref{figure9} presents the values of $I(\bar{M}^k, \bar{m}_{\mathrm{ref}})$ for various choices of the step size $\delta_k = k_2/(k + k_1)$.
The approximate convergence orders for each case are reported in Table~\ref{table2}.
Since the supremum of the expected order $s/r$ is attained when $r = 2$ and $s = k_2$, these numerical results are consistent with Theorem~\ref{I23_thm3.5}.
However, the estimated convergence orders are slightly lower than those observed in the one-dimensional case.
This observation indicates that the convergence rate of the GCG method differs between the one-dimensional and two-dimensional settings, as stated in Theorem~\ref{LP23_thm8-2}.

\begin{table}[t]
\caption{Estimated convergence orders corresponding to each step-size (Example \ref{ex:2D}).}
\centering
\label{table2}
\begin{tabular}{c||c|c|c} \hline
 $\delta_k=\frac{k_2}{k+k_1}$ & $k_1=1, k_2=1$ & $k_1=2, k_2=2$ & $k_1=2, k_2=1$ \\ \hline
 $I(\bar{M}^k, \bar{m}_{\mathrm{ref}})$ & $\mathcal{O}\left((k+1)^{-0.798}\right)$ & $\mathcal{O}\left((k+2)^{-1.408}\right)$ & $\mathcal{O}\left((k+2)^{-0.757}\right)$ \\ \hline
\end{tabular}

\begin{tabular}{c||c|c|c} \hline
 $\delta_k=\frac{k_2}{k+k_1}$ & $k_1=3, k_2=3$ & $k_1=3, k_2=2$ & $k_1=3, k_2=1$\\ \hline
 $I(\bar{M}^k, \bar{m}_{\mathrm{ref}})$ & $\mathcal{O}\left((k+3)^{-1.966}\right)$ & $\mathcal{O}\left((k+3)^{-1.403}\right)$ & $\mathcal{O}\left((k+3)^{-0.829}\right)$ \\ \hline
\end{tabular}
\end{table}

\section{The case of the global coupling terms}
\label{sec:gct}

Lavigne and Pfeiffer \cite{LP23} studied the GCG method under the following condition on the coupling term $f$: 
\begin{description} 
\item[(f-L2)] There exist constants $C_0$, $L_f > 0$ and $\alpha_0 \in (0, 1)$ satisfying: 
\begin{equation*}
|f(t_2, x_2, m_2) - f(t_1, x_1, m_1)| \le C_0 \left(
|t_2-t_1|^{\alpha_0} +|x_2-x_1|\right)+ L_f \|m_2 - m_1\|_{L^{2}(\T^d)}
\end{equation*}
for $(t_1, x_1),(t_2, x_2) \in Q$, and $m_1,m_2 \in \mathcal{D}_0(\T^d)$. 
\end{description}

An example of $f$ satisfying (f-B), (f-L2), (f-M), and (f-P) can be found in \cite[Lemma 2]{B21}; see also \cite[Example 1.1]{MR3127145} for a related construction.

The quantity $r$ appearing \eqref{eq:err-1} of Theorem \ref{I23_thm3.5} and \eqref{eq:err-2} of Theorem \ref{I23_thm3.5a} depends on the spatial dimension $d$. In particular, $r$ should be strictly greater than $d$. 
However, if assuming \textup{(f-L2)} instead of the original condition \textup{(f-L)}, we can take $r=2$ in those estimates. 

To show this, we now reorganize our argument in a step-by-step manner under this assumption. 

\begin{itemize}
  \item 
  Propositions~\ref{B21_thm1} and~\ref{LP23_prop23} are identical to \cite[Theorem~3]{LP23} and \cite[Proposition~23]{LP23}, respectively, and can be used without modification.

  \item It is proved in \cite[Theorem 23]{LP23} that 
\begin{equation*}
\|\bar{m}_{k}-\bar{m}\|_{{L^\infty(0, T; L^2(\T^d))}}
+\|u_k-\bar{u}\|_{L^{\infty}(Q)} 
+\|\gamma_k - \bar{\gamma}\|_{L^{\infty}(Q)} 
\le C \eps_k^{{\frac{1}{2}}}.  
\end{equation*}
This inequality should be compared with \eqref{eq:thm7} of Proposition \ref{LP23_thm7}. In fact, $\bar{m}_{k}-\bar{m}$ and $\gamma_k - \bar{\gamma}$ are estimated using the norms of $L^\infty(0,T;L^2(\T^d))$ and $L^\infty(Q)$ instead of $L^2(0,T;L^\infty(\T^d))$ and $L^2(0,T;L^\infty(\T^d))$, respectively. 
This affects the way the error is measured in Theorem \ref{I23_thm3.5z}.

\item 
Proposition~\ref{LP23_thm8-2} also holds in this case.
In particular, a careful examination of the proof of \cite[Theorem~3.8]{NS26}
shows that, for any spatial dimension \(d\ge 1\), one may take
$s = k_2$ and $N = \varepsilon_0 k_1^{k_2}
\exp\!\left(\frac{C_{\ast} k_2^2 (k_1+1)}{k_1^2}\right)$.

  \item The inequalities \eqref{eq:e51} and \eqref{eq:e52} in Proposition~\ref{I23_prop4.4,4.6} can be used without modification.
Indeed, in the proof, assumption \textup{(f-L)} is used only in \eqref{ineq_estimate_r4_1}, which remains valid under assumption \textup{(f-L2)} because
$$
\left|f(t_n,x_i,\bar{m}_k(t_n))-f(t_n,x_i,\bar{M}^k_n)\right|
\le L_f\|\bar{m}_k(t_n)-\bar{M}^k_n\|_{L^2(\mathbb{T}^d)}
\le L_f\|\bar{m}_k(t_n)-\bar{M}^k_n\|_{L^\infty(\mathbb{T}^d)}.
$$
Proposition~\ref{I23_thm3.4} can also be applied without any change.
\end{itemize}

Taking the above remarks into account, the proof of Theorem \ref{I23_thm3.5} can be adapted to obtain the following result, stated as Theorem~\ref{I23_thm3.5z}. We set 
\begin{equation*}
\tilde{I}(\phi,\psi) := \max_{0\le n\le N_T-1}
\left[
\sum_{i\in\Lambda_x} |\phi_{n,i} - \psi(t_n,x_i)|^2 \Delta x\right]^{\frac{1}{2}}
\end{equation*}
for a grid function $\phi_{n,i}$ in $\Lambda$ and continuous function $\psi$ in $Q$. 

\begin{theorem} 
\label{I23_thm3.5z}
Suppose that 
\textup{(H)}, 
\textup{(f-B)}, \textup{(f-M)}, \textup{(f-L2)}, \textup{(f-P)}, and \textup{(TIV)} are satisfied. 
The initial guess $\bar{m}_0$ is the function specified in Proposition~\ref{prop:initial}. 
The step-size $\delta_k$ is chosen as \eqref{eq:predefine}. 
Then, there exist a constant $\eta\in (0,1)$, functions $C_{\mathrm{q}}^{(1)}(k) > 0$ ($\mathrm{q} = \bar{m}, \bar{u}$) of $k$, and constants $C_{\mathrm{q}}^{(2)} > 0$ ($\mathrm{q} = \bar{m}, \bar{u}$) such that
\begin{align*}
\tilde{I}(\bar{M}^k, \bar{m}) &\le C_{\bar{m}}^{(1)}(k) \left[(\Delta t)^{\frac{\eta}{2}} + (\Delta x)^{\eta}\right] + \frac{C_{\bar{m}}^{(2)}}{(k+k_1)^{k_2/2}}, \\
E(U^k,\bar{u}) &\le C_{\bar{u}}^{(1)}(k) \left[(\Delta t)^{\frac{\eta}{2}} + (\Delta x)^{\eta}\right] + \frac{C_{\bar{u}}^{(2)}}{(k+k_1)^{k_2/2}}.
\end{align*}
The asymptotic growth of $C_{\mathrm{q}}^{(1)}(k)$ ($\mathrm{q} = \bar{m}, \bar{u}$) 
are polynomially bounded from above and logarithmically bounded from below as $k\to\infty$. Furthermore, if $\phi_k$ and $\psi_k$ have the regularity property 
\eqref{eq:C24}, $(\Delta t)^{\frac{\eta}{2}} + (\Delta x)^{\eta}$ can be replaced by $\Delta t+\Delta x$. 
\end{theorem}

\begin{rem}
The convergence order $(\Delta t)^{\frac{\eta}{2}} + (\Delta x)^{\eta}$ cannot be replaced by $\Delta t+(\Delta x)^2$, since \eqref{ineq_estimate_r4_2} does not follow under \textup{(f-L2)}. 
\end{rem}

\begin{rem}
It should be kept in mind that the local coupling term \eqref{eq:ex-f} does not satisfy \textup{(f-L2)}. 
\end{rem}

\section{Concluding remarks}
\label{sec:remarks}

\begin{enumerate}
\item \emph{Summary.} In this paper, we studied the GCG method for the MFG system at the fully discrete level.
By reformulating the GCG method via the Cole--Hopf transformation and discretizing the resulting system with finite difference schemes, we derived a numerical method that preserves essential qualitative properties such as a discrete maximum principle. Our main contribution is the derivation of explicit convergence estimates that simultaneously quantify the discretization error and the iteration error. In particular, we showed that the total error can be decomposed into a mesh-dependent term and an iteration-dependent term, and we clarified how these two sources of error interact. The resulting convergence rates are explicit with respect to the time step, the spatial mesh size, and the iteration number.
We also demonstrated that, under additional regularity assumptions, higher-order convergence rates with respect to the discretization parameters can be achieved.
Although the convergence with respect to the iteration is not uniform, our results rigorously justify the convergence of the discrete GCG scheme and provide theoretical guidance for practical parameter choices.

\item \emph{Extension to general Hamiltonians.} 
In this paper, we have restricted our attention to a quadratic Hamiltonian with the advection effect \eqref{eq:quadH}. 
This restriction is imposed in order to reduce the system to a reaction–diffusion system via the Cole--Hopf transformation.
This feature is also exploited in an essential way in the previous work \cite{NS26}.
However, although the GCG method and the CH–GCG method are equivalent at the PDE level through the Cole–Hopf transformation, the system obtained by directly discretizing the GCG method is, in general, not equivalent to the discrete GCG scheme, and a discrepancy arises between the two.
Consequently, a more natural approach would be to solve the HJB equation without relying on linearization techniques such as the Cole–Hopf transformation.
Such an approach would also make it possible to handle more general Hamiltonians.
In fact, achieving this extension would require both a new regularity theory for HJB equations and detailed error estimates for numerical methods for such equations. If these were established, the results of the present paper could potentially be extended to more general Hamiltonians. These are important open problems in analysis and numerical analysis.

\item \emph{Variational structure of discrete GCG scheme.} 
Under assumption \textup{(f–P)}, both the MFG system and the GCG method possess a variational structure.
Although this property is not used directly in the present paper, the variational structure plays a role in the proofs of Proposition~\ref{LP23_thm7} and Proposition~\ref{LP23_thm8-2}.
On the other hand, we have not yet clarified whether the discrete GCG scheme admits a corresponding variational structure, and if so, what its precise form would be.
In fact, analyzing the convergence of the discrete GCG scheme based on such a variational structure may lead to more natural error estimates, or even to the development of improved discretization methods.
Understanding the variational structure of the discrete GCG scheme (as well as the discrete MFG system) therefore remains an interesting open question.

This viewpoint may also be important for the extension to general Hamiltonians discussed in item 2. In a direct discretization of the PDE-level GCG iteration \eqref{GCG}, the HJB equation is expected to produce an error-propagation structure similar to that in the present analysis, with the density-iterate error entering through the coupling term. The main additional difficulty arises in the Fokker--Planck equation: its drift depends on $\nabla u_k$, which must be replaced by a discrete gradient $\nabla_{\sharp} U^k$. A direct estimate of this gradient error may introduce a factor of order $1/\Delta x$ and prevent the resulting error bound from converging as the mesh is refined.

A possible alternative approach is to formulate the GCG iteration directly for a discrete MFG system using its variational structure. This would allow the convergence of the discrete iteration $(\bar{M}^k, U^k)$ to the discrete MFG solution $(\bar{M}, \bar{U})$ and the discretization error between the discrete and continuous MFG systems to be studied separately. Developing such a discrete variational framework may therefore be important for extending the scheme and its error analysis to general Hamiltonians.

\item \emph{Boundary conditions and other discretization methods.} In this paper, we have consistently assumed periodic boundary conditions when analyzing the MFG system and the GCG method.
However, from the viewpoint of practical applications, it is often more natural to consider the system on a bounded domain and to impose appropriate boundary conditions.
Moreover, depending on the geometry of the domain and the choice of boundary conditions, it may be preferable to adopt numerical methods such as finite element or finite volume schemes rather than finite difference methods.
For instance, when applying finite element methods, the validity of a discrete maximum principle is not guaranteed a priori and typically requires additional technical considerations. 
These topics are left for future investigation.

\item \emph{Adaptive step-size selection methods.} 
In this paper, we focus on the predefined step-size selection \eqref{eq:predefine}. 
More sophisticated step-size strategies are known to yield faster convergence; see \cite{LP23,NS26}. 
While convergence of the discrete GCG scheme can be proved for any step-size sequence with \(\varepsilon_k \to 0\), deriving explicit convergence rates requires additional estimates on \(\varepsilon_k\). 
Such strategies, however, typically increase the computational complexity due to the need to evaluate functionals during implementation.
\end{enumerate}

\appendix 
\section{A remark on the size of $\overline{C}_2$}
\label{sec:app}

In this appendix, we consider the case $d=1$. 
The constants $\Phi_{\max}$, $\Phi_{\min}$, and $\Psi_{\max}$ in Theorem \ref{thm:DMP}, the constants $\phi_{\max}$ and $\phi_{\min}$ in Proposition \ref{prop:phi_max_min}, and a constant $\psi_{\max}$ satisfying
$$
0\le \psi_k \le \psi_{\max}
$$
may be chosen explicitly as follows:
\begin{align*}
\Phi_{\max}&=\exp\left(\frac{\|g\|_{L^\infty(\T)}}{2\nu}\right),\\
\Phi_{\min}&=\exp\left(-\frac{T\|f\|_{L^\infty(Q)}+\|g\|_{L^\infty(\T)}}{2\nu}\right),\\
\Psi_{\max}&=\frac{e^{TL_h}\|m_0\|_{L^\infty(\T)}}{\Phi_{\min}}=\|m_0\|_{L^\infty(\T)}\exp\left(TL_h+\frac{T\|f\|_{L^\infty(Q)}+\|g\|_{L^\infty(\T)}}{2\nu}\right),\\
\phi_{\max}&=\exp\left(\frac{\|g\|_{L^\infty(\T)}+T\|f\|_{L^\infty(Q)}}{2\nu}\right),\\
\phi_{\min}&=\exp\left(-\frac{\|g\|_{L^\infty(\T)}+T\|f\|_{L^\infty(Q)}}{2\nu}\right),\\
\psi_{\max}&=\frac{e^{TL_h}\|m_0\|_{L^\infty(\T)}}{\phi_{\min}}.
\end{align*}
Here, $L_h$ denotes the Lipschitz constant of $h$ with respect to $x$. Moreover, the constants $C_2$ and $C_4$ in Proposition \ref{I23_prop4.4,4.6} are given by
\begin{align*}
C_2 &= \frac{T}{2\nu} \phi_{\max},\\
C_4 &= C_2\frac{e^{TL_h}\|m_0\|_{L^\infty(\T)}}{\phi_{\min}\Phi_{\min}}+\frac{e^{TL_h}-1}{2\nu L_h}L_f\psi_{\max}.
\end{align*}
When $L_h=0$, the factor $(e^{TL_h}-1)/(2\nu L_h)$ is understood as $T/2\nu$.

The estimates for $\phi_{\max}$ and $\phi_{\min}$ are obtained by observing that
$$
\overline{u}(t,x):=\|g\|_{L^\infty(\T)}+(T-t)\|f\|_{L^\infty(Q)}
$$
and
$$
\underline{u}(t,x):=-\|g\|_{L^\infty(\T)}-(T-t)\|f\|_{L^\infty(Q)}
$$
are, respectively, a supersolution and a subsolution of \eqref{GCGa} with the terminal condition \eqref{GCGb}. This implies that $\|u_k\|_{L^{\infty}(Q)} \le \|g\|_{L^{\infty}(\T)} + T \|f\|_{L^{\infty}(Q)}$, from which the upper and lower bounds for $\phi_k = \exp{\left(-u_k/2 \nu\right)}$ follow immediately. Moreover, $\widetilde{\psi}(t,x):=e^{tL_h}\|m_0\|_{L^\infty(\T)}/\phi_{\min}$ is a supersolution of \eqref{CH_MFG_eq-3} with \eqref{CH_MFG_eq-4}; hence, the upper bound $\psi_{\max}$ for $\psi_k$ follows from the comparison principle. For further details on the estimates of $\Phi_{\max}$, $\Phi_{\min}$, $\Psi_{\max}$, and $C_4$, we refer to \cite[Propositions 4.3, 4.5, and 4.6]{I23}.

Since $\phi$ and $\psi$ are nonnegative functions, using $\|\phi_k\|_{L^\infty(Q)}\le \phi_{\max}$ and $\|\psi_k\|_{L^\infty(Q)}\le \psi_{\max}$, we obtain an explicit upper bound for $\bar{C}_2$. Setting $S:=(\|g\|_{L^\infty(\T)}+T\|f\|_{L^\infty(Q)})/2\nu$, we have
\begin{align*}
\bar{C}_2&=C_2\Psi_{\max}+C_4\Phi_{\max}\\
&\le \frac{T}{2 \nu} e^{TL_h}\|m_0\|_{L^\infty(\T)}e^S\left[e^S+\left(e^{2S}+\frac{e^{TL_h}-1}{T L_h}L_f\right)\exp\left(\frac{\|g\|_{L^\infty(\T)}}{2\nu}\right)\right].
\end{align*}
Since the above bound contains several exponential factors involving $e^S$, verifying the condition $\bar{C}_2<1$ appears to be difficult. For example, for the parameters used in Example 6.1 of Section \ref{sec:numerical_experiment}, the resulting upper bound is approximately $1.09\times10^{147}$. Even if the data were replaced by more favorable ones, for instance by taking $\nu$ larger and choosing $\|f\|_{L^\infty(Q)}$ and $\|g\|_{L^\infty(\T)}$ to be relatively small, the condition $\bar{C}_2<1$ would still not be easy to guarantee from this bound. We emphasize, however, that the above estimate is not expected to be sharp. A refinement of the estimates may therefore make it possible to establish $\bar{C}_2<1$ under substantially less restrictive assumptions on the data.

\section*{Acknowledgements}
We would like to thank Dr. Daisuke Inoue (Toyota Central R{\&}D Labs., Inc.) and Professor Takahito Kashiwabara (The University of Tokyo)
for their valuable advice and insightful discussions during the course of this research. 
We also thank the anonymous reviewers for their valuable comments and suggestions to improve the quality of the paper.
This work was partially supported by JSPS KAKENHI Grant Number 21H04431 and 26H02001 (Grant-in-Aid for Scientific Research (A)). Nakamura was supported by the SPRING GX program of The University of Tokyo.

\printbibliography

\end{document}